\documentclass[a4paper,12pt,reqno]{amsart}
\usepackage[T1]{fontenc}
\usepackage[utf8]{inputenc}
\usepackage{amssymb,dsfont,mathrsfs}
\usepackage{mathtools}
\usepackage[margin=1in]{geometry}
\usepackage{marginnote} 
\usepackage{enumitem}
\usepackage{xparse}
\usepackage[bookmarksdepth=2]{hyperref}
\usepackage[mathscr]{eucal}
\usepackage{xcolor, tikz}

\numberwithin{equation}{section}

\theoremstyle{plain}
\newtheorem{thm}{Theorem}

\newtheorem{lem}[thm]{Lemma}
\newtheorem{prop}[thm]{Proposition}
\newtheorem{cor}[thm]{Corollary}
\newtheorem{rem}[thm]{Remark}
\newtheorem{ass}[thm]{Assumption}
\theoremstyle{definition}
\newtheorem{ex}[thm]{Example}

\numberwithin{thm}{section} 

\newcommand{\R}{\mathbb{R}}
\newcommand{\Rn}{{\R^n}}
\newcommand{\BE}{\mathcal{B}(E)}
\newcommand{\EBE}{(E,\BE)}
\newcommand{\BEstar}{\mathcal{B}^*(E)}
\newcommand{\BeE}{\mathcal{B}^e(E)}

\newcommand{\EAm}{(m)}
\newcommand{\LpE}{L^p\EAm}
\newcommand{\N}{\mathbb{N}}
\newcommand{\C}{C}
\newcommand{\E}{\mathcal{E}}
\newcommand{\D}{\mathcal{D}}
\newcommand{\DAp}{\D(A_p)}
\newcommand{\DE}{\D(\E)}
\newcommand{\EDE}{(\E,\DE )}

\newcommand{\F}{\mathcal{F}}
\newcommand{\Prb}{\mathbb{P}}
\newcommand{\Px}{\Prb_{x}}

\newcommand{\Cem}{\Delta}
\newcommand{\excm}{\tilde{m}}
\newcommand{\Pexcm}{\Prb_{\excm}}
\newcommand{\Pt}{(P_t)}

\newcommand{\Pbrevet}{(\st{P}_t)}
\newcommand{\Pbstt}{(\bst{P}_t)}
\newcommand{\Xt}{X}
\newcommand{\Xbrevet}{\st{X}}
\newcommand{\Xbstt}{\bst{X}}
\newcommand{\Ft}{(\F_t)}
\newcommand{\Mt}{M}
\newcommand{\At}{A}
\newcommand{\Bt}{B}
\newcommand{\Pxx}{(\Px)}
\newcommand{\intb}{\int\limits}
\newcommand{\iintb}{\iint\limits}
\newcommand{\intE}{\intb_E}
\newcommand{\intRn}{\intb_\Rn}
\newcommand{\tim}{\times}
\newcommand{\setms}{\setminus}
\newcommand{\intEx}{\intb_{\hspace{1em}\mathclap{E\setms\{x\}}}}
\newcommand{\intEz}{\intb_{\hspace{1em}\mathclap{E\setms\{z\}}}}
\newcommand{\intEXs}{\intb_{\hspace{1em}\mathclap{E\setms\{X_s\}}}}

\newcommand{\diag}{d}
\newcommand{\EEd}{E\times E\setms\diag}

\newcommand{\iintEEd}{\iintb_{\mathclap{\EEd}}}
\newcommand{\iintqEEd}{\iintb_{\hspace{0.7em}\mathclap{\EEd}}}

\newcommand{\dd}{d}
\newcommand{\dms}{\,\dd m}

\newcommand{\dxs}{\,\dd x}
\newcommand{\dys}{\,\dd y}
\newcommand{\dts}{\,\dd t}
\newcommand{\dss}{\,\dd s}
\newcommand{\dx}{\dd x}
\newcommand{\dy}{\dd y}
\newcommand{\dz}{\dd z}
\newcommand{\dt}{\dd t}
\newcommand{\ds}{\dd s}
\newcommand{\dm}{\dd m}
\newcommand{\dr}{\dd r}
\newcommand{\dtm}{\dd \excm}
\newcommand{\dxdy}{(\dx,\dy)}

\newcommand{\inft}{{\infty}}
\newcommand{\ph}{\varphi}

\newcommand{\Gampri}{\widetilde{\Gamma}}
\newcommand{\Ev}{\mathds{E}}
\DeclareMathOperator{\ind}{\mathds{1}}
\DeclareMathOperator*\sgn{sgn}
\DeclareMathOperator*\supp{supp}
\DeclareMathOperator*\Capa{Cap}
\newcommand{\ignore}[1]{}
\newcommand{\pow}[1]{^{\langle #1\rangle}}
\newcommand{\abs}[1]{\left\lvert #1 \right\rvert}
\newcommand{\Bnorm}[1]{\left\lVert #1 \right\rVert}
\newcommand{\BnormLp}[1]{\Bnorm{ #1 }_{p}}

\newcommand{\norm}[1]{\mathopen{\lVert} #1 \mathclose{\rVert}}
\newcommand{\normLp}[1]{\norm{ #1 }_{p}}
\newcommand{\normL}[2]{\norm{ #1 }_{ #2 }}
\newcommand{\LE}[1]{L^{ #1 } \EAm}

\newcommand{\st}[1]{\check{ #1 }}
\newcommand{\bst}[1]{\hat{ #1 }}
\newcommand{\subpr}[1]{{ #1 }^\alpha}
\newcommand{\equi}[1]{\overset{ #1 }{\simeq}}

\author[M.~Gutowski]{Micha\l\ Gutowski }
\title{Littlewood--{Paley} estimates for pure-jump {Dirichlet} forms }

\thanks{The research was supported by National Science Centre, Poland, 2023/49/B/ST1/04303.}

\begin{document}

\begin{abstract}
    We employ the recent generalization of the Hardy--Stein identity to extend the previous Littlewood--Paley estimates to general pure\nobreakdash-jump Dirichlet forms.
    The results generalize those for symmetric pure-jump L\'evy processes in Euclidean spaces.
    We also relax the assumptions for the Dirichlet form necessary for the estimates used in previous works.
    To overcome the difficulty that It\^o's formula is not applicable, we employ the theory of Revuz correspondence and additive functionals.
    Meanwhile, we present a few counterexamples demonstrating that some inequalities do not hold in the generality considered in this paper.
    In particular, we correct errors that appear in previous works.
\end{abstract}

\maketitle

\section{Introduction}
\label{sec:introduction}

Consider a regular pure\nobreakdash-jump Dirichlet form $\EDE$ and let $\Xt$ be the associated symmetric pure\nobreakdash-jump Hunt process. By $J$ we denote the jumping kernel of $\EDE$ (or $\Xt$). Let $\Pt$ be the associated semigroup. In \cite{bbl}, Ba{\~n}uelos, Bogdan, and Luks investigated the following \emph{square functions} (or \emph{Littlewood\nobreakdash--Paley functions}):
\begin{align}
    \label{eq:G_fun}
    G(x) & := \left(
        \frac{1}{2} \int\limits_0^\inft \intEx ( P_tf(y) -  P_tf(x))^2 \,J(x,\dy)\dt
    \right)^{1/2}
    , \\ \label{eq:Gprim_fun}
    \widetilde{G}(x) & := 
    \left(
        \int\limits_0^\inft \intEx ( P_tf(y) -  P_tf(x))^2 \,\chi( P_tf(x), P_tf(y))J(x,\dy)\dt
    \right)^{1/2}.
\end{align}
Here, $f$ is a fixed real-valued function in the $L^p$\nobreakdash-space for $1<p<\infty$, and $\chi(s, t)=\ind_{\{|s| > |t|\}} + \tfrac{1}{2} \ind_{\{|s| = |t|\}}$. The authors of \cite{bbl} studied the \emph{Littlewood\nobreakdash--Paley estimates} of those square functions. That is, they investigated the existence of constants $c_p, C_p>0$ such that
\begin{align}
\label{eq:Littlewood-Paley-est-intro}
    c_p \normLp{f}
    \le
    \normLp{G}
    \le
    C_p \normLp{f},
\end{align}
\begin{align}
\label{eq:Littlewood-Paley-est2-intro}
    c_p \normLp{f}
    \le
    \normLp{\widetilde{G}}
    \le
    C_p \normLp{f}.
\end{align}

While the function $G$ seems more natural, it lacks an upper bound; see Example~2 of \cite{bbl}.
Therefore, the authors also considered its modification $\widetilde{G}$.
It turns out that if $X$ is a symmetric pure-jump L\'evy process satisfying the Hartman--Wintner condition, then $\widetilde{G}$ satisfies \eqref{eq:Littlewood-Paley-est2-intro} for some $c_p, C_p>0$ and for the entire range of $1<p<\infty$.

In \cite{bbl}, two distinct techniques were employed to establish this result: the Hardy--Stein identity and the Burkholder--Davies--Gundy inequality.
The latter involves the parabolic martingale $M_t := P_{T-t}f(X_t) - P_t(X_0)$.
The authors applied It\^o's formula to relate its quadratic variation $\langle M \rangle$ to the third square function:
\begin{align}
    \label{eq:H_fun}
    H(x) & := \left(
        \frac{1}{2} \int\limits_0^\inft \intE \intEz ( P_tf(y) -  P_tf(z))^2 \,J(z,\dy)P_t(x,\dz)\dt
    \right)^{1/2}.
\end{align}
In the next step, the authors established the upper bound
$\normLp{H}
\le
C_p \normLp{f}$
for $2 \le p <\infty$. Next, they demonstrated the inequality $\widetilde{G}\le 2 H$ utilizing the ``time doubling''; see the proof of Lemma~4.2 in \cite{bbl}.
Finally, by combining these observations, they derived a similar upper bound for $\widetilde{G}$.
The lower bound for $1<p<2$ was derived employing the ultracontractivity of the semigroup associated with a L\'evy process.

In Li and Wang \cite{lw}, an attempt was made to apply a similar approach under weaker assumptions; however, the reasoning contains an error.
It turns out that in the greater generality, the inequality $\widetilde{G}\le 2 H$ does not hold. In \cite{lw}, the proof of this inequality relies on an unjustified assumption that the  ``time doubling'' still can be employed and that $P_t[P_tf(k(\cdot,y))](x)$ is equal to $P_{2t}f(k(x,y))$ for some function $k$. For more details, we refer to Remark~\ref{rem:GlessH} below.
In the present work, we provide a counterexample demonstrating that the above error is irreparable; see Example~\ref{cex:Brown} below.

The Hardy\nobreakdash--Stein identity in a non\nobreakdash-local case was proved in \cite{bbl}, where symmetric pure\nobreakdash-jump L\'evy processes were considered. See also the non\nobreakdash-symmetric extension in Ba{\~n}uelos and Kim \cite{bk} as well as the polarized version presented in Bogdan, Gutowski, and Pietruska\nobreakdash-Pa{\l}uba \cite{bgp}.
See also Bogdan, Kutek, and Pietruska-Pa{\l}uba~\cite{bkp}.
The Hardy\nobreakdash--Stein identity was originally established by Stein in his book \cite[pp. 86--88]{stein-singular} for the Laplace operator. His approach was based essentially on the chain rule:
$\Delta u^p = p(p-1)u^{p-2}|\nabla u|^2+pu^{p-1}\Delta u$.
Stein applied this identity to establish Littlewood\nobreakdash--Paley estimates for square functions associated with the Laplace operator.

Recent papers derive a generalization of the Hardy--Stein identity.
In this context, we refer to Gutowski \cite{g} and Gutowski and Kwaśnicki \cite{gk}.
The main goal of the present work is to employ the aforementioned generalization to extend previous Littlewood\nobreakdash--Paley estimates for a broader class of pure-jump processes (and Dirichlet forms).
While \cite{lw} suggests that:
\emph{It seems that such an identity depends heavily on the characterisation of L\'evy processes, and may not hold for general jump processes} (see p.~424 therein),
our results demonstrate that the Hardy--Stein identity remains valid in a more general setting.
The aforementioned generalization allows us to extend previous results not just to an extended class of Markov processes on the Euclidean space $\Rn$, but also to a wider variety of topological spaces $E$.

We want to focus on a general pure\nobreakdash-jump Dirichlet form possessing a jumping kernel.
Consequently, It\^o's formula is not applicable, and we do not assume the ultracontractivity of the semigroup. Instead, we employ the Revuz correspondence to establish a link between the quadratic variation $\langle M \rangle$ and the square function $H$.
Nevertheless, the reasoning presented in this article draws extensively on the approach outlined in \cite{bbl}.

The Revuz correspondence describes the connection between additive functionals and a certain class of Radon measures.
In particular, in the case of positive continuous additive functionals (abbreviated as PCAF) this correspondence is a bijection.
The Revuz correspondence has been the subject of extensive study. Notable contributions include the works of Getoor \cite{Getoor}, \cite{gs_Naturality}, as well as Fitzsimmons and Getoor \cite{fg88, fg_smooth}. Additionally, we refer to Section 75 of Sharpe’s book \cite{Sharpe}.
Moreover, a relatively recent paper Li and Ying \cite{BivariateRevuz} extends previous results to the case of bivariate measures.

Since the parabolic martingale $\Mt$ is an additive functional with respect to the space\nobreakdash-time process $(t,X_t)$,
its quadratic variation $\langle M \rangle$ is a PCAF.
In the present paper, we characterize the Revuz correspondence in this specific case; see Lemma~\ref{lem:RevuzST} in the next section.

Summarizing the results presented in this paper, we obtain the following Littlewood\nobreakdash--Paley estimates for $\widetilde{G}$ in the generality we consider. In particular, the previously mentioned Example~\ref{cex:Brown} demonstrates that, in general, these estimates cannot be established for the entire range of $1<p<\infty$.
\begin{thm}
    Let $1<p<\infty$ and $f\in\LpE$. Under the strong stability assumption (i.e. $\lim_{T\to\inft} \normLp{P_Tf} = 0$)
    there are constants $c_p, C_p>0$ such that
    \begin{align*}
        c_p \normLp{f}
        \le
        \normLp{\widetilde{G}},
        \quad
        2\le p< \infty,
    \end{align*}
    and
    \begin{align*}
        \normLp{\widetilde{G}}
        \le
        C_p \normLp{f},
        \quad
        1<p\le 2.
    \end{align*}
    However, there does not exist a universal constant $\tilde{C}_p>0$ such that
    \begin{align}
    \label{neq:Gprim_unbound}
        \normLp{\widetilde{G}} \le \tilde{C}_p \normLp{f},
        \quad
        2<p<\infty.
    \end{align}
\end{thm}

In light of the impossibility of obtaining the estimate for the square function $\widetilde{G}$, we then proceed to investigate the function $H$ and obtain the following results.
\begin{thm}
    Let $1<p<\infty$ and $f\in\LpE$. Under the strong stability assumption
    and the conservativeness assumption (i.e. $P_t1=1$ a.e.)
    there are constants $c_p, C_p>0$ such that
    \begin{align*}
        c_p \normLp{f}
        \le
        \normLp{H},
        \quad
        3\le p<\infty,
    \end{align*}
    and
    \begin{align*}
        \normLp{H}
        \le
        C_p \normLp{f},
        \quad
        2\le p< \infty.
    \end{align*}
    Additionally, if $f\in\LpE\cap\LE{2}$, then the constant $c_p$ can be chosen so that
    \begin{align*}
        c_p \normLp{f}
        \le
        \normLp{H},
        \quad
        1<p\le 2.
    \end{align*}
    However, there does not exist a universal constant $\tilde{C}_p>0$ such that
    \begin{align}
    \label{neq:Hprim_unbound}
        \normLp{H}
        \le
       \tilde{C}_p \normLp{f},
        \quad
        1< p< 2.
    \end{align}
\end{thm}

For the same reason that no upper bound exists for the function $G$, there is also no upper estimate for the function $H$; see Example~\ref{cex:Cauchy2} below.
Therefore, at the later stage we study the fourth square function given by
\begin{align}
    \label{eq:Hprim_fun}
    \widetilde{H}(x) & := \left(
        \int\limits_0^\inft \intE \intEz ( P_tf(y) -  P_tf(z))^2 \,\chi( P_tf(z), P_tf(y))J(z,\dy)P_t(x,\dz)\dt
    \right)^{1/2}.
\end{align}

In the present work, we derive the following Littlewood\nobreakdash--Paley estimates for the function $\widetilde{H}$.
\begin{thm}
    Let $1<p<\infty$ and $f\in\LpE$. Under the strong stability assumption
    there is a constant $c_p>0$ such that
    \begin{align*}
        c_p \normLp{f}
        \le
        \normLp{\widetilde{H}},
        \quad
        3\le p< \infty.
    \end{align*}
    Additionally, if we assume the conservativeness, then there is a constant $C_p>0$ such that
    \begin{align*}
        \normLp{\widetilde{H}}
        \le
        C_p \normLp{f},
        \quad
        2\le p < \infty.
    \end{align*}
\end{thm}
At this point, the question of whether it is possible to obtain estimates for the remaining ranges of $1<p<\infty$ remains open. 
It is unclear which methods should be employed in this context.
Approaches relying on the Hardy--Stein identity seem particularly suitable for the functions $G$ and $\widetilde{G}$, whereas martingale techniques appear to be more effective for the function $H$.

Square functions were first introduced by Littlewood and Paley \cite{LittlewoodPaleyI}, although their origins can be traced back to the works of Kaczmarz \cite{Kaczmarz} and Zygmund~\cite{Zygmund1927}; see also Zygmund~\cite{Zygmund_Remarque_Kaczmarz}.
These early studies explored the use of square functions in analyzing the pointwise convergence of Fourier series.
Subsequently, Littlewood and Paley developed a series of results linking square functions to Fourier series across multiple papers \cite{LittlewoodPaleyI, LittlewoodPaleyII, LittlewoodPaleyIII, PaleyI, PaleyII}. Over time, further research expanded both on their work and on entirely new types of square functions. Notable contributions include Marcinkiewicz~\cite{Marcinkiewicz38}, who introduced the so-called Marcinkiewicz square function, and Marcinkiewicz and Zygmund~\cite{Marcinkiewicz_Zygmund_lusin}, who studied Lusin’s area integral. Zygmund \cite{Zygmund_powerI, Zygmund_powerII} investigated square functions in the context of power series, while Marcinkiewicz~\cite{Marcinkiewicz39} applied them to Fourier multipliers.
Further advancements were made by Zygmund in \cite{Zygmund44}, where he established the $L^p$-boundedness
of the Marcinkiewicz square function, and in \cite{Zygmund45}, where he provided a simplified proof of the results originally presented by Littlewood and Paley in \cite{LittlewoodPaleyII}.

Subsequent research conducted by Zygmund and his students paved the way for new developments in the theory of square functions. At this stage, the theory was extended to the multidimensional setting. For references, see Calderón~\cite{Calderon50}, Stein~\cite{Stein58}, and Benedek, Calderón, and Panzone~\cite{Benedek_Calderon_Panzone62}; see also Calderón~\cite{Calderon50b}.
Additionally, substantial progress was made in the analysis of the Marcinkiewicz square function, as discussed in Weiss and Zygmund~\cite{Weiss_Zygmund59} and Stein~\cite{Stein58}. Notably, this function found applications in the study of derivatives, including fractional ones, in the $L^p$ sense; see Stein and Zygmund~\cite{Stein_Zygmund64, Stein_Zygmund65}.
Furthermore, we would like to highlight the contributions of Calderón and Zygmund~\cite{Calderon_Zygmund52}, as well as Zygmund's own works~\cite{Zygmund56, Zygmund59}.

Later advancements in the study of square functions played a key role in the development of the theory of Hardy spaces. Notable contributions include Burkholder’s work~\cite{Burkholder66}, which extended Paley's theorem from Walsh-Paley series to general martingales, followed by further generalizations by Burkholder and Gundy~\cite{Burkholder_Gundy70} and Burkholder, Gundy, and Silverstein~\cite{Burkholder_Gundy_Silverstein71}.
Additionally, we highlight the works of Stein~\cite{Stein66a, Stein66b, Stein67}, Fefferman and Stein~\cite{Fefferman_Stein72}, and Gundy and Stein~\cite{Gundy_Stein79}, where Hardy spaces were, in particular, characterized using Brownian motion.

The classical Littlewood--Paley theory has found numerous other applications, including those in the works of Stein~\cite{Stein76}, Stein and Wainger~\cite{Stein_Wainger78}, Nagel, Stein, and Wainger~\cite{Nagel_Stein_Wainger78}, Weiss and Wainger~\cite[pp.~429--434]{Guido_HA2},
Fabes, Jerison, and Kenig~\cite{MR815765}, and
Jones and Kenig~\cite[pp.~24--90]{MR1013813}.
Many of these ideas have been incorporated into Stein’s books~\cite{stein-singular, stein-topics}. Notably, in~\cite{stein-topics}, square functions played a fundamental role in proving a maximal inequality for semigroups; see Lemma~\ref{lem:Stein_ineq} below.
This result is widely used throughout the present work.
For a more comprehensive overview of the historical development of classical Littlewood--Paley theory, as summarized in the previous paragraphs, we refer the reader to Stein's outstanding essay~\cite{Stein_essay}.

At this stage, it is worth mentioning that, given the wide range of applications of square functions, the literature contains a diverse selection of such functions -- some closely related, while others are more distantly connected.
Nevertheless, the scope of our article focuses on the square functions associated with stochastic processes.
The first proposal to connect processes with a broad class of square functions was introduced by Meyer in his series of articles \cite{MeyerIIb, MeyerI, MeyerII, MeyerIII, MeyerIV}. We refer also to later corrections \cite{MeyerIV_corr, Meyer_corr}.
In essence, this method is based on expressing a square function by the carr\'{e} du champ operator $\Gamma$ of the process. This operator was introduced by Meyer in \cite{MeyerIIb}.


Specifically, the carr\'{e} du champ operator associated with the Brownian motion is $\Gamma[u]=\abs{\nabla u}^2$. Therefore, the square functions investigated by Stein in \cite{stein-singular} can be understood as those associated with Brownian motion and, consequently, with the classical Laplace operator~$\Delta$.
This topic was further carried out by Meyer in \cite{Meyer84, Meyer85}.

An alternative probabilistic perspective on the Littlewood--Paley functions from \cite{stein-topics} was suggested by Varopoulos~\cite{Varopoulos80}. Furthermore, we draw attention to the work of Ba{\~n}uelos~\cite{Banuelos86}, which presented the Brownian counterpart of the Lusin area integral. See also the monograph by Ba{\~n}uelos and Moore~\cite{Banuelos_Moore99}. However, closely related square functions for Brownian motion were introduced somewhat earlier by Bennett in~\cite{Bennett85}. Later, Bouleau and Lamberton extended the previously mentioned probabilistic methods from Brownian motion to the $\alpha$-stable processes in \cite{Bouleau_Lamberton86}.
Further exploration of the Littlewood--Paley theory in the $\alpha$-stable case is provided in recent works, such as those by Kim and Kim~\cite{Kim_Kim2012} and Karl{\i}~\cite{Karki2013}.

The remainder of the paper is organized as follows.
In the next section, we introduce all the required notions and facts.
Firstly, we describe a pure\nobreakdash-jump Dirichlet form, associated semigroup, the symmetric Hunt process, and its carr\'{e} du champ operator.
In Subsection~\ref{sub:dev}, we establish some auxiliary facts concerning $L^p$-derivatives, which are useful for the connection between the carr\'{e} du champ operator and the quadratic variation of a martingale, as presented later in Section~\ref{sec:Sharp_bracket}.
Subsection~\ref{sub:dev} contains a formulation of the Hardy--Stein identity in the context of the pure\nobreakdash-jump case, along with the useful inequalities that follow from it.
In Subsection~\ref{sub:square_functions}, we formulate square functions utilizing the carr\'{e} du champ operator.
Subsection~\ref{sub:Revuz} is dedicated to the Revuz correspondence.
In Subsection~\ref{sub:space_time} we define the space\nobreakdash-time process derived from the initial symmetric Hunt process.
Subsection~\ref{sub:martingale} presents the basic facts concerning martingales, their quadratic variation, and states the Burkholder--Davies--Gundy inequality.

In Section~\ref{sec:Sharp_bracket}, we prove a formula for the quadratic variation in terms of carr\'{e} du champ operator.
Section~\ref{sec:square_fun} is dedicated to the Littlewood--Paley estimates of square functions, where we provide the proofs of $L^p$\nobreakdash-bounds and present some counterexamples.

\section{Preliminaries}
\label{sec:prelim}

In this paper, we work under the standard topological configuration for Dirichlet forms; see (1.1.7) in Fukushima, Oshima, and Takeda \cite{fot}.
More precisely, we consider a locally compact separable metric space $E$ and denote the $\sigma$\nobreakdash-algebra of all Borel sets in $E$ by $\BE$. Let $m$ be a Radon measure on $E$ of full support, that is, $\supp{m}=E$.

Let $1\le p\le \infty$. The real-valued Banach space $L^p(E,\BE,m)$, equipped with the norm $\normLp{\cdot}$, will be referred to as $\LpE$. Let $\C(E)$ be the class of continuous functions on $E$. By $\C_c(E)$ we denote the class of functions from $\C(E)$ with compact support.

We frequently use the notation
$f(x) \lesssim g(x)$ (resp. $f(x) \gtrsim g(x)$)
to indicate that there exists a positive constant $C_p$ depending only on $p$
such that $f(x) \le C_p g(x)$ (resp. $f(x) \ge C_p g(x)$)
for all relevant arguments $x$.
If both
$f(x) \lesssim g(x)$ and $f(x) \gtrsim g(x)$
hold,
we denote this by $f(x) \asymp g(x)$.
Additionally, when two sequences $(a_n)_{n\in\N}$ and $(b_n)_{n\in\N}$ of non\nobreakdash-negative real numbers are asymptotically equal, that is, $a_n / b_n \to 1$ as $n\to\inft$, we write $a_n \sim b_n$. Similarly, for non\nobreakdash-negative functions $f$ and $g$, we will write $f(x)\sim g(x)$ as $x\to x_0$ whenever $\lim_{x\to x_0} f(x) /g(x)=1$ and say that those functions are asymptotically equal as $x\to x_0$.

Consider a regular Dirichlet form $\EDE$ with $\DE\subseteq \LE{2}$. We assume that $\EDE$ is pure\nobreakdash-jump, that is, it consists only of the jumping part in the Beurling\nobreakdash--Deny decomposition; see Lemma 4.5.4 in \cite{fot}. In other words, the form $\E$ is given by the following formula:
\begin{align}
\label{eq:E-app}
    \E(u,v)
    =
    \frac{1}{2} \iintEEd (u(y) - u(x)) (v(y) - v(x)) \,J\dxdy
\end{align}
for all quasi-continuous $u,v\in\DE$.
Here, $J$ is the \emph{jumping measure} of $\EDE$
and
$\diag:=\{(x,y)\in E\times E: x=y\}$ is the diagonal of $E\times E$.
Further, we always assume that we choose the quasi\nobreakdash-continuous version of $u\in\DE$ without mentioning; see Theorem~2.1.7 in \cite{fot}.
We will also write $\E[u]:=\E(u,u)$.

It should be noted that, due to our convention, the measure $J$ considered here is a factor of two larger than the one in \cite{fot}. Compare \eqref{eq:E-app} with \cite[(4.5.13)]{fot}.
This choice is inspired by the probabilistic interpretation of the jumping measure $J$. We refer here to~\cite[(5.3.6)]{fot}. Moreover, it is consistent with the usage found in other works.

Throughout the paper $X=(\Omega, \F, \Ft, (\theta_t), (X_t), (\Prb_x))$ will denote a symmetric pure\nobreakdash-jump Hunt process associated with the Dirichlet form $\EDE$.
Then the semigroup $\Pt$ of $\Xt$ extends to a semigroup on $\LpE$ ($1 \le p \le \infty$) which is a semigroup of sub\nobreakdash-Markovian contractions on $\LpE$ ($1 \le p \le \infty$), strongly continuous on $\LpE$ ($1 \le p < \infty$), symmetric on $\LE{2}$, and satisfies
\begin{align}
\label{eq:E(t)E}
    \left\{\begin{aligned}
        \DE &= \left\{u\in\LE{2}: \text{finite } \lim_{t\to0^+} \frac{1}{t}\int_E (u-P_tu)u \dms \text{ exists}\right\},\\
        \E(u,v) &= \lim_{t\to0^+} \frac{1}{t}\int_E (u-P_tu)v \dms,\quad u,v\in\DE.\\
    \end{aligned}\right.
\end{align}
For the correspondence between symmetric Hunt processes and regular Dirichlet forms, we refer to Sections~4.2 and 7.2 in \cite{fot}. For the correspondence between regular Dirichlet forms and the semigroups on $\LE{2}$ we refer to Section~1.3 in \cite{fot}. For the $\LpE$-extension of $\Pt$ we refer to \cite[p.~38, p.~56]{fot} and Section~1 of Farkas, Jacob, and Schilling \cite{fjs}. See also Section~2.6 of the monograph by Jacob \cite[pp.~133--138]{jacob_vol_2}. The last two references assume the setting $E=\Rn$, but the same argument applies for general $E$.

For such a class of semigroups, the following maximal inequality of Stein is valid. We refer to \cite[p. 73]{stein-topics}.
\begin{lem}[Stein's maximal inequality]
\label{lem:Stein_ineq}
    Let $1 < p \le \infty$ and $f\in\LpE$. Denote
    $f^*(x):=\sup_{t\ge0}|P_tf(x)|$.
    Then, there exists a constant $D_p>0$ such that,
    \begin{align}
    \label{neq:Stein}
        \normLp{f^*}
        \le
        D_p \normLp{f}.
    \end{align}
\end{lem}

\begin{rem}
    The above maximal inequality may also be derived from Doob's maximal inequality for martingales.
    For this, see the discussion in \cite[pp.~106--107]{stein-topics}, where Rota's theorem was used.
    For this reason, in the literature, the above constant $D_p$ is commonly taken to be equal to $\tfrac{p}{p-1}$ if $1<p<\infty$ and $1$ if $p=\infty$.
\end{rem}

Let $A_p$ be the \emph{infinitesimal generator} of the semigroup $\Pt$ on $\LpE$:
\begin{align}
\label{eq:Lpdef}
    A_p u &:= \lim_{t\to0^+} \frac{1}{t}(P_tu - u)
    \quad
    \text{in }
    \LpE,
\end{align}
with the natural domain
\begin{align*}
    \DAp := \{u\in\LpE: \lim_{t\to0^+} \tfrac{1}{t}(P_tu - u) \text{ exists in }\LpE\}.
\end{align*}

Denote the lifetime of the process $\Xt$ by $\zeta$.
Let $\Cem$ be the cemetery state of $\Xt$. By $E_\Cem$ we denote the one\nobreakdash-point compactification of $E$.
When $E$ is already compact, then $\Cem$ is an isolated point of $E_\Cem$.
Since $\Xt$ is symmetric, it possesses the \emph{self-dual} property with respect to the reference measure $m$. This implies that, for any $T>0$ and a non\nobreakdash-negative measurable function $\Phi$ defined on the space of trajectories,
\begin{align}
\label{eq:self-dual}
    \intE \Ev_x \Phi((X_{(T-t)-} : t \in [0, T]))\ind_{\{T<\zeta\}}\,m(\dx) & = \intE \Ev_x \Phi((X_t : t \in [0, T]))\ind_{\{T<\zeta\}}\,m(\dx) .
\end{align}
In particular, for a non\nobreakdash-negative Borel function $\ph\colon [0,\infty)\tim E\to\R$, we may write
\begin{align}
\label{eq:dualityZeta}
    \intE \Ev_x \left(\int\limits_0^T \ph(T - t, X_t)\dts
    \right)
    \ind_{\{T<\zeta\}}
    \,m(\dx) 
    &=
    \intE \Ev_x \left(\int\limits_0^T \ph(t, X_{T - t})\dts\right)
    \ind_{\{T<\zeta\}}
    \,m(\dx)
    \notag \\
    & = \intE \Ev_x \left(\int\limits_0^T \ph(t, X_t)\dts\right)
    \ind_{\{T<\zeta\}}
    \,m(\dx).
\end{align}
For some of the results we will need the following assumption.
\begin{ass}[Conservativeness]
\label{ass:coserv}
    $
        P_t1=1
    $ a.e.
    for every $t>0$.
\end{ass}
If we impose the above assumption, then the identity \eqref{eq:dualityZeta} takes the following form
\begin{align}
\label{eq:duality}
    \intE \Ev_x \left(\int\limits_0^T \ph(T - t, X_t)\dts
    \right)
    \,m(\dx) 
    = \intE \Ev_x \left(\int\limits_0^T \ph(t, X_t)\dts\right)
    \,m(\dx).
\end{align}

We assume additionally that the jumping measure $J$ is given by the kernel $J(x,\dx)$, i.e.,
\begin{align}
\label{eq:JK}
    J\dxdy = J(x,\dy)m(\dx).
\end{align}
This assumption is not too restrictive. Indeed, for a general jumping measure $J$ we may always write
$J\dxdy = N(x,\dy)\nu(\dx)$,
where $(N,H)$ is the L\'evy system of $\Xt$ and $\nu$ is the Revuz measure of $H$; see Theorem~5.3.1 in \cite{fot}. The notion of the Revuz measure is described in Subsection~\ref{sub:Revuz}.

We want to stress that, due to our convention, $J$ is equal to $N(x,\dy)\nu(\dx)$, hence it is two times greater than the one appearing in \cite{fot}.
The motivation comes from the probabilistic meaning of the jumping measure (see (5.3.6) in~\cite{fot}); this convention also aligns with those adopted by several other authors.

Under assumption \eqref{eq:JK}, the regular Dirichlet form $\EDE$ possesses a \emph{carr\'{e} du champ operator} given by:
\begin{align}
\label{eq:carre}
    \Gamma[u](x) := \frac{1}{2} \intEx (u(y) - u(x))^2 \,J(x,\dy).
\end{align}
Since the integrand is non\nobreakdash-negative, the above operator is well\nobreakdash-defined for every measurable function $u$, possibly infinite.
In particular, $\Gamma[u]$ is finite a.e. for $u\in\DE$.

It is easy to see that for any $u\in\DE\cap\LE{\infty}$, the function $\Gamma[u]$ is the unique element of $\LE{1}$ such that
\begin{align*}
    \intE f \Gamma[u] \dms = \E(uf,u) - \tfrac{1}{2}\E(f,u^2),
    \quad
    f \in \DE\cap\C_c(E).
\end{align*}
Therefore, in light of our convention, the function $\Gamma[u]$ is two times smaller than the one in \cite{fot} and \cite{bh_book}. To see this, compare the above formula with \cite[(3.2.14)]{fot} and \cite[Proposition~4.1.3]{bh_book}.

It is known that the above operator is continuous as a mapping $\Gamma\colon \DE\to\LE{1}$, where the space $\DE$ is equipped with the seminorm $\sqrt{\E[\cdot]}$, that is,
\begin{align}
\label{neq:carreCon}
    \normL{\Gamma[u]}{1} \le \E[u]
    , \quad u\in\DE;
\end{align}
see Proposition~4.1.3 in Bouleau and Hirsch \cite{bh_book}.
Additionally, when $u\in\D(A_2)$, then $u^2 \in \D(A_1)$ and
\begin{align}
\label{eq:gamma_DA2}
    \Gamma[u] = \tfrac{1}{2} A_1(u^2) - u A_2u;
\end{align}
see Theorem~4.2.2 in \cite{bh_book}.

We consider also a modified version of the operator $\Gamma$:
\begin{align}
\label{eq:Gampri}
    \Gampri[u](x) := \intEx (u(y) - u(x))^2 \,\chi(u(x),u(y)) \,J(x,\dy),
\end{align}
where
\begin{align}
\label{eq:chi}
    \chi(s, t) & :=
    \begin{cases}
        1 & \text{if } |s| > |t| , \\
        \tfrac{1}{2} & \text{if } |s| = |t| , \\
        0 & \text{if } |s| < |t| .
    \end{cases}
\end{align}
Note that $\chi$ is the characteristic function of the set $\{(s,t): |s| > |t|\}$, with a minor modification when $|s| = |t|$.
Since $\chi(s, t) + \chi(t, s) = 1$, from the symmetry of the jumping measure $J$ we have
\begin{align}
\label{eq:Gamma+-}
    \intE \Gamma[u] \dms = \intE \Gampri[u] \dms = \E[u],
    \quad
    u\in\DE.
\end{align}

\subsection{Derivatives on $L^p$}
\label{sub:dev}

Let $1\le p<\infty$ and $I\subseteq\R$ be some interval. Given a mapping
$I \ni t \mapsto u(t) \in \LpE$ we denote
\begin{align*}
    \Delta_hu(t) := u(t+h)-u(t) \quad \text{if }t,t+h\in I.
\end{align*}

We call $u$ \emph{continuous} on $I$ with values in $\LpE$ if $\Delta_hu(t) \to 0$ in $\LpE$ as $h\to 0$ for every $t\in I$.
A function $u$ is called \emph{differentiable} on $I$ with values in $\LpE$ if the limit $u'(t) := \lim_{h \to 0} \frac{1}{h} \Delta_hu(t)$ exists in $\LpE$ for every $t\in I$.
We say that $u$ is \emph{continuously differentiable} (or shortly $\C^1$) on $I$ with values in $\LpE$ if $u$ is differentiable and the mapping $I \ni t \mapsto u'(t) \in \LpE$ is continuous.

The idea of employing $L^p$-derivatives was motivated by the works \cite{bjlp, bgp}.
In \cite{bjlp}, $L^p$-derivatives were applied to the study of contractivity of perturbed semigroups; see the proof of Theorem~3 therein.
In \cite{bgp}, the authors used $L^p$-derivatives to prove the polarized Hardy--Stein identity. For multidimensional $L^p$\nobreakdash-calculus, we refer to \cite[Appendix~B]{bgp}.

In the following part of this subsection, we establish some auxiliary facts concerning $L^p$-derivatives. They will play a role in the proof of the main result of Section~\ref{sec:Sharp_bracket}.

\begin{lem}
\label{lem:dif_Ptut}
    Let $1<p<\infty$. Let $u$ be differentiable on $I\subseteq[0,\inft)$ with values in $\LpE$. Then $P_tu(t)$ is differentiable on $I\setms\{0\}$ and
    \begin{align}
    \label{eq:dif_Ptut_for_p}
        (P_tu(t))' = A_p P_tu(t) + P_t u'(t),
        \quad
        t\in I\setms\{0\}.
    \end{align}
   Additionally, under the assumption that for each $t\in I$, $P_tu(t)\in\D(A_1)$, then $P_tu(t)$ is differentiable on $I$ with values in $\LE{1}$ and
    \begin{align}
    \label{eq:dif_Ptut}
        (P_tu(t))' = A_1 P_tu(t) + P_t u'(t),
        \quad
        t\in I.
    \end{align}
\end{lem}
\begin{proof}
    Let $1\le p<\infty$. Observe that
    \begin{align*}
        \frac{1}{h} \Delta_h (P_tu(t))
        =
        P_{t+h} u'(t) + P_{t+h}\left(\frac{1}{h} \Delta_hu(t) - u'(t)\right)
        +
        \frac{1}{h} (P_{t+h}-P_t)u(t).
    \end{align*}
    Due to the strong continuity of $\Pt$, the first term converges to $P_tu'(t)$ as $h\to0$.
    The second term converges to zero, because
    \begin{align*}
        \BnormLp{P_{t+h}\left(\frac{1}{h} \Delta_hu(t) - u'(t)\right)}
        \le
        \BnormLp{\frac{1}{h} \Delta_hu(t) - u'(t)} \to 0,
        \quad
        \text{as }h\to0.
    \end{align*}
    The above inequality follows from the contraction property of $\Pt$.

    The last part converges to $A_p P_tu(t)$ by the definition of the generator $A_p$.
    Indeed, when $p>1$, the semigroup $\Pt$ is analytic on $\LpE$; see \cite[Theorem~1 in Chapter~II]{stein-topics}. Hence, for any $f\in\LpE$ and $t>0$, $P_tf\in\DAp$. In particular, $P_tu(t)\in\DAp$, when $t\neq0$.
    In the case of $p=1$, the convergence follows from the assumption $P_tu(t)\in\D(A_1)$.
    In summary of the above observations,
    \begin{align*}
        \frac{1}{h} \Delta_h (P_tu(t)) \to P_t u'(t) + 0 + A_p P_tu(t) 
    \end{align*}
    in $\LpE$ as $h\to0$.
\end{proof}
\begin{rem}
    For $1<p<\infty$, if we assume that $u(0)\in\DAp$, then $P_tu(t)$ is differentiable on $I$ and \eqref{eq:dif_Ptut_for_p} holds for all $t\in I$.
\end{rem}

\begin{lem}
\label{lem:difPtPTmtf2}
    Let $f\in\LE{2}$ and let $T>0$. The mapping $[0,T)\ni t \mapsto P_t[(P_{T-t}f)^2]$ is $\C^1$ with values in $\LE{1}$ and
    \begin{align*}
        (P_t[(P_{T-t}f)^2])' = 2P_t(\Gamma[P_{T-t}f]).
    \end{align*}
\end{lem}
\begin{proof}
    Denote $u(t):=(P_{T-t}f)^2$. Due to \eqref{eq:dif_Ptut} from Lemma~\ref{lem:dif_Ptut}
    \begin{align*}
        (P_t[(P_{T-t}f)^2])' = A_1  P_t[(P_{T-t}f)^2] + P_t u'(t).
    \end{align*}
    Let us justify why the use of this theorem is valid.
    For every $t\in[0,T)$, $P_{T-t}f\in\D(A_2)$ due to the fact that the semigroup $\Pt$ is analytic on $\LE{2}$. Therefore, $(P_{T-t}f)^2\in\D(A_1)$ and also $P_t[(P_{T-t}f)^2]\in\D(A_1)$; see \cite[Theorem~4.2.2]{bh_book}.

    According to \cite[Lemma~B.3(i)]{bgp},
    $u'(t) = 2P_{T-t}f(P_{T-t}f)' = -2P_{T-t}f A_2P_{T-t}f$.

    Bringing everything together,
    \begin{align*}
        (P_t[(P_{T-t}f)^2])'
        &=
        A_1  P_t[(P_{T-t}f)^2]
        -
        2 P_t [P_{T-t}f A_2P_{T-t}f]
        \\
        &=
        P_t [A_1[(P_{T-t}f)^2] - 2 P_{T-t}f A_2P_{T-t}f]
        \\
        &=
        2P_t(\Gamma[P_{T-t}f]).
    \end{align*}
    The last line follows from \eqref{eq:gamma_DA2}.

    In the end, we show that the above derivative is continuous with values in $\LE{1}$. Fix $t_0\in[0,T)$. We may write
    \begin{align*}
        \normL{P_t(\Gamma[P_{T-t}f]) - P_{t_0}(\Gamma[P_{T-t_0}f])}{1}
        &\leq
        \normL{P_t(\Gamma[P_{T-t}f]) - P_t(\Gamma[P_{T-t_0}f])}{1}
        \\
        &\quad+
        \normL{P_t(\Gamma[P_{T-t_0}f]) - P_{t_0}(\Gamma[P_{T-t_0}f])}{1}
    \end{align*}
    The second term converges to zero as $t\to t_0$ by the strong continuity of $\Pt$.
    
    Since the semigroup $\Pt$ is a contraction on $\LE{1}$, the first term can be estimated as follows,
    \begin{align*}
        \normL{P_t(\Gamma[P_{T-t}f]) - P_t(\Gamma[P_{T-t_0}f])}{1}
        &\leq
        \normL{\Gamma[P_{T-t}f] - \Gamma[P_{T-t_0}f]}{1}
        \\
        &\leq
        \left(\sqrt{\E[P_{T-t}f]} + \sqrt{\E[P_{T-t_0}f]}\right) \sqrt{\E[P_{T-t}f - P_{T-t_0}f]}
    \end{align*}
    The last inequality follows from the continuity of $\Gamma$; see \eqref{neq:carreCon}.
    As the semigroup $\Pt$ is strongly continuous on $\DE$ equipped with the seminorm $\sqrt{\E[\cdot]}$ (see Lemma~1.3.3(iii) in \cite{fot}), the right-hand side tends to zero as $t\to t_0$.
\end{proof}

\subsection{Hardy--Stein identity}

Following the approach from Ba{\~n}uelos, Bogdan, and Luks \cite{bbl}, we use the Hardy\nobreakdash--Stein identity to obtain some of the Littlewood\nobreakdash--Paley estimates. In the setting presented above, this identity is given by the following formula:
\begin{align}
\label{eq:HS-app}
    \intE |f(x)|^p \,m(\dx)
    -
    \lim_{T\to\inft} \normLp{P_Tf}^p
    &=
    \int\limits_0^\inft \iintqEEd
    F_p (P_tf(x),P_tf(y)) \,J\dxdy\dt
    .
\end{align}
Here, $f\in\LpE$, where $1<p<\infty$,
\begin{align}
\label{eq:Fp}
    F_p(a,b)
    :=
    |b|^p-|a|^p - pa\pow{p-1}(b-a)
\end{align}
is the \emph{Bregman divergence}, and $a\pow{\gamma}:=\abs{a}^\gamma \sgn a$ is the \emph{French power}. Note that $F_p$ is the second\nobreakdash-order Taylor remainder of the convex function $\R\ni a\mapsto |a|^p\in\R$. In particular, $F_p$ is non\nobreakdash-negative.
The Hardy\nobreakdash--Stein identity holds for every symmetric Hunt process (not only pure\nobreakdash-jump). For this, we refer to Corollary~1.2 in Gutowski and Kwaśnicki \cite{gk} and to Gutowski \cite{g}.

For the Bregman divergence $F_p$ the following estimate holds
\begin{align}
\label{sim:Fp_sim_Gp}
    F_p(a,b) \asymp \abs{b - a}^2 (\abs{a}\vee \abs{b})^{p-2}
    \asymp
    \abs{b - a}^2 (\abs{a} + \abs{b})^{p-2},
    \quad
    a,b\in\R.
\end{align}
Hence, if $p\ge2$, then also
\begin{align}
\label{sim:Fp_sim_Gp_p2}
    F_p(a,b)
    \asymp
    \abs{b - a}^2 (\abs{a}^{p-2} + \abs{b}^{p-2}),
    \quad
    a,b\in\R.
\end{align}
For $1<p<2$ we only have
\begin{align}
\label{sim:Fp_sim_Gp_p3}
    F_p(a,b)
    \lesssim
    \abs{b - a}^2 (\abs{a}^{p-2} + \abs{b}^{p-2}),
    \quad
    a,b\in\R.
\end{align}
The first estimate in \eqref{sim:Fp_sim_Gp} (in its vectorized form) was derived by Pinchover, Tertikas, and Tintarev in \cite{pyt}; see~(2.19).
However, earlier one-sided bounds may be found in Shafrir \cite[Lemma 7.4]{Shafrir} and, for $2\le p < \infty$, in Barbatis, Filippas, and Tertikas \cite[Lemma~3.1]{bft}.
The one\nobreakdash-dimensional case was presented also by Bogdan, Dyda, and Luks in \cite{bdl}; see Lemma~6. We also refer to Bogdan, Grzywny, Pietruska-Pa{\l}uba, and Rutkowski \cite[Lemma~2.3]{bgpr}.
For optimal constants in specific ranges of $p$, we refer to \cite[Lemma 7.4]{Shafrir} and Bogdan, Jakubowski, Lenczewska, and Pietruska\nobreakdash-Pałuba \cite{bjlp}.

To associate the right-hand side of the Hardy--Stein identity \eqref{eq:HS-app} with the carr\'{e} du champ operator $\Gamma$, we use estimate \eqref{sim:Fp_sim_Gp_p2}. Precisely, if $p\ge 2$, then we can write
\begin{align}
    \iintqEEd
    &F_p (P_tf(x),P_tf(y)) \,J\dxdy
    \notag \\
    & \asymp
    \iintqEEd (P_tf(y) - P_tf(x))^2 (|P_tf(x)|^{p - 2} + |P_tf(y)|^{p - 2}) \,J\dxdy
    \notag \\
    & =
    2 \iintqEEd (P_tf(y) - P_tf(x))^2 |P_tf(x)|^{p - 2} \,J\dxdy
    \notag \\ \label{eq:iiFp_p2}
    & =
    \intE \Gamma[P_tf](x) |P_tf(x)|^{p - 2} \,m(\dx).
\end{align}
For $1< p < 2$, we only have $\lesssim$-bound in the above estimate, as given in \eqref{sim:Fp_sim_Gp_p3}.

Similarly, we treat the operator $\Gampri$. For any $1<p<\infty$, by estimate \eqref{sim:Fp_sim_Gp},
\begin{align}
    \iintqEEd
    &F_p (P_tf(x),P_tf(y)) \,J\dxdy
    \notag \\
    & \asymp
    2 \iintqEEd (P_tf(y) - P_tf(x))^2 (|P_tf(x)|\vee|P_tf(y)|)^{p - 2} \,J\dxdy
    \notag \\
    & =
    \iintqEEd (P_tf(y) - P_tf(x))^2 |P_tf(x)|^{p - 2} \chi(P_tf(x), P_tf(y)) \,J\dxdy
    \notag \\ \label{eq:iiFp_p1}
    & =
    \intE \Gampri[P_tf](x) |P_tf(x)|^{p - 2} \,m(\dx).
\end{align}

Combining the above observations with the Hardy--Stein identity, we obtain the following estimate between the $p$-norm of the function $f$ and the carr\'{e} du champ operator $\Gamma$ and its modification $\Gampri$.
When $1< p \le 2$, then we get
\begin{align}
\label{neq:HS_12_bezSS}
    \intE |f(x)|^p\,m(\dx)
    & \gtrsim
    \int\limits_0^\inft \intE \Gampri[P_tf](x) |P_tf(x)|^{p - 2} \,m(\dx)\dt
    ,
\end{align}
while, if $2\le p < \infty$, then
\begin{align*}
    \intE |f(x)|^p\,m(\dx)
    & \gtrsim
    \int\limits_0^\inft \intE \Gampri[P_tf](x) |P_tf(x)|^{p - 2} \,m(\dx)\dt
    \\
    & \asymp
    \int\limits_0^\inft \intE \Gamma[P_tf](x)  |P_tf(x)|^{p - 2} \,m(\dx)\dt
    .
\end{align*}

For some of Littlewood\nobreakdash--Paley estimates, we need the following not overly restrictive assumption.
\begin{ass}[Strong stability]
\label{ass:SS}
    $
        \lim_{T\to\inft} \normLp{P_Tf} = 0
    $
    for every $f\in\LpE$.
\end{ass}
When the semigroup $\Pt$ is conservative (see Assumption~\ref{ass:coserv}) and the jumping measure satisfies some irreducibility assumption, then the limit $\lim_{T\to\inft} \normLp{P_Tf}$ is equal to $\normLp{\Bar{f}}$, where $\Bar{f}$ is the mean value of $f$ when $m(E)<\inft$ and $\Bar{f}=0$ when $m(E)=\inft$. For more details, see the discussion in Appendix A of \cite{g}.

Without Assumption~\ref{ass:SS}, most of the lower Littlewood\nobreakdash--Paley estimates are not valid. Indeed, if there exists some non\nobreakdash-zero $g\in\LpE$ such that $\lim_{T\to\inft} \normLp{P_Tg} > 0$, then, for $f:=\lim_{T\to\inft} P_Tg$, we have $P_tf=f$ and $\normLp{f}>0$. It is known that the inner integral of the right\nobreakdash-hand side of the Hardy\nobreakdash--Stein identity \eqref{eq:HS-app} equals
\begin{align*}
    \iintqEEd
    F_p (P_tf(x),P_tf(y)) \,J\dxdy
    =
    \lim_{h\to0^+} \frac{p}{h}\intE (P_tf-P_hP_tf)(P_tf)\pow{p-1} \dms.
\end{align*}
See the notion of the Sobolev\nobreakdash--Bregman form in \cite{gk}.
Since $P_tf=P_hP_tf=f$, the right\nobreakdash-hand side is equal to zero. Hence, by \eqref{eq:iiFp_p1},
\begin{align*}
    \intE \Gampri[P_tf](x) |P_tf(x)|^{p - 2} \,m(\dx) = 0.
\end{align*}
It is easy to see that this implies $\Gampri[P_tf]=0$ a.e.

Imposing Assumption~\ref{ass:SS}, we can strengthen estimates between $\normLp{f}$ and $\Gamma$ and $\Gampri$ as follows.
Then, if $1< p \le 2$, then
\begin{align*}
    \intE |f(x)|^p\,m(\dx)
    & \asymp
    \int\limits_0^\inft \intE \Gampri[P_tf](x) |P_tf(x)|^{p - 2} \,m(\dx)\dt
    \\
    & \lesssim
    \int\limits_0^\inft \intE \Gamma[P_tf](x) |P_tf(x)|^{p - 2} \,m(\dx)\dt
    ,
\end{align*}
while, if $2\le p < \infty$, then
\begin{align}
\label{sim:HS_2pinf}
    \intE |f(x)|^p\,m(\dx)
    & \asymp
    \int\limits_0^\inft \intE \Gampri[P_tf](x) |P_tf(x)|^{p - 2} \,m(\dx)\dt
    \\ \notag
    & \asymp
    \int\limits_0^\inft \intE \Gamma[P_tf](x)  |P_tf(x)|^{p - 2} \,m(\dx)\dt
    .
\end{align}

\subsection{Square functions}
\label{sub:square_functions}

Fix $f\in\LpE$, where $1<p<\infty$. The main scope of this work is to study the \emph{square functions} (or \emph{Littlewood--Paley functions}) introduced in Section~\ref{sec:introduction}; see equations \eqref{eq:G_fun}, \eqref{eq:Gprim_fun}, \eqref{eq:H_fun}, and \eqref{eq:Hprim_fun}. Utilizing notions of operators $\Gamma$ and $\Gampri$ defined by \eqref{eq:carre} and \eqref{eq:Gampri} we can rewrite earlier formulas of square functions in the following way:
\begin{align*}
    G(x) & := \left(
        \int\limits_0^\inft \Gamma [P_tf](x) \dts
    \right)^{1/2},
    \\
    \widetilde{G}(x) & := 
    \left(
        \int\limits_0^\inft \Gampri[P_tf](x) \dts 
    \right)^{1/2},
    \\
    H(x) & :=
    \left(
        \int\limits_0^\inft P_t \Gamma [P_tf](x) \dts 
    \right)^{1/2},
    \\
    \widetilde{H}(x) & :=
    \left(
        \int\limits_0^\inft P_t \Gampri[P_tf](x)\dts 
    \right)^{1/2}.
\end{align*}

Due to the fact that $\Gampri[u]\leq2\Gamma[u]$ a.e., we immediately obtain $\widetilde{G} \le \sqrt{2} G$ and $\widetilde{H} \le \sqrt{2} H$.
In view of~\eqref{eq:Gamma+-}, the $L^2$\nobreakdash-norms of $G$, $\widetilde{G}$, $H$, and $\widetilde{H}$ are all equal to
\begin{align*}
    \left(
        \int\limits_0^\inft \intE \Gamma [P_tf] \dms\dt
    \right)^{1/2}.
\end{align*}
Impose for a moment Assumption~\ref{ass:SS}.
Observe that for $p=2$ the Bregman divergence is equal to $F_2(a,b) = (b-a)^2$. Therefore, as a straightforward consequence of the Hardy\nobreakdash--Stein identity, we obtain the following fact.
\begin{prop}
    Impose Assumption~\ref{ass:SS}. Let $f\in\LE{2}$. Then, the following equivalence between $L^2$\nobreakdash-norms holds:
    \begin{align*}
        \normL{G}{2} = \normL{\widetilde{G}}{2} = \normL{H}{2}
        = \normL{\widetilde{H}}{2}
        =
        \tfrac{1}{\sqrt{2}}\normL{f}{2}.
    \end{align*}
\end{prop}

\subsection{Additive functionals and Revuz correspondence}
\label{sub:Revuz}

In this subsection, let us assume that the process 
$X=(\Omega, \F, \Ft, (\theta_t), (X_t), (\Prb_x))$
is an arbitrary Hunt process, not necessarily pure-jump or symmetric. Assume that $\Ft$ is the minimum completed admissible filtration.

Recall that $\Pt$ is the semigroup associated with the process $\Xt$. For each $t\ge0$, an operator $P_t$ has its adjoint operator $\hat{P}_t$ on $\LE{2}$, i.e.,
$\langle P_t f, g \rangle = \langle f, \hat{P}_t g \rangle$,
for all $f,g\in\LE{2}$. Clearly, $P_t=\hat{P}_t$ for symmetric $\Xt$.
By
\begin{align*}
    R_\alpha f := \int\limits_0^\inft e^{-\alpha t} P_tf \dts,
    \qquad
    \hat{R}_\alpha f := \int\limits_0^\inft e^{-\alpha t} \hat{P}_tf \dts,
    \quad
    \alpha>0,
\end{align*}
we denote the \emph{resolvent} (resp. \emph{coresolvent}) of $\Xt$. Here, the above integrals are understood in the Bochner sense in $\LE{2}$.
The resolvent and coresolvent are strongly continuous, that is for $f\in\LE{2}$, we have $\alpha R_\alpha f \to f$ and $\alpha \hat{R}_\alpha f \to f$ in $\LE{2}$ as $\alpha\to\inft$.

Let $\alpha\geq0$. A non\nobreakdash-negative Borel function $f$ is \emph{$\alpha$\nobreakdash-excessive} (resp. \emph{$\alpha$\nobreakdash-coexcessive}) with respect to the semigroup $\Pt$, if $e^{-\alpha t} P_tf\nearrow f$ (resp.
$e^{-\alpha t} \hat{P}_tf \nearrow f$) as $t\searrow 0$.
Let $f$ be a non\nobreakdash-negative Borel function. A simple example of an $\alpha$\nobreakdash-excessive (resp. $\alpha$\nobreakdash-coexcessive) function is $\alpha R_\alpha f$ (resp. $\alpha \hat{R}_\alpha f$).
We say that a Borel measure $\excm$ on the space $\EBE$ is \emph{excessive} if it is $\sigma$\nobreakdash-finite and $\int_E P_t\ind_B \dtm \le \excm(B)$, $B\in\BE$.
In particular, if the semigroup $\Pt$ is symmetric, then the measure $m$ introduced in Section~\ref{sec:prelim} is excessive.

Let $\excm$ be a Borel measure on $\EBE$.
For a measure $\mu$ on $\EBE$ we denote $\Prb_\mu (\Lambda) := \int_E \Px(\Lambda) \,\mu(\dx)$.
Recall that we say that a set $B\subseteq E$ is \emph{nearly Borel} (relative to $\Xt$) if, for each probability measure $\mu$, there exist Borel sets $B_1, B_2 \subseteq E$ such that $B_1\subseteq B \subseteq B_2$ and
$\Prb_\mu (X_t\in B_2\setms B_1\text{ for some }t)=0$; see Definition (10.21) in Blumenthal and Getoor~\cite{blumenthal_getoor} or~\cite[p.~392]{fot}.
Before we introduce the notion of additive functionals, we need to provide some definitions of sets which are ``small'' or ``big'' with respect to the process $\Xt$ and measure $\excm$.
Denote by $\sigma_B := \inf\{t>0: X_t\in B\}$ the \emph{hitting time} of a set $B$ for $\Xt$.
We say that a nearly Borel set $D\subseteq E$ is \emph{$\excm$-polar} if
$\Pexcm(\sigma_D<\inft)=0$;
see Definition~(6.3) in Getoor and Sharpe \cite{gs_Naturality}.
A nearly Borel set $\widetilde{E}\subseteq E$ is \emph{absorbing} if $\Px(\sigma_{E\setms\widetilde{E}}<\inft) = 0$ for all $x\in\widetilde{E}$; see \cite[p.~17]{gs_Naturality}.
Finally, we say that a nearly Borel set $N\subseteq E$ is \emph{$\excm$-inessential} if $N$ is $\excm$-polar and $N^c$ is absorbing; see Definition~(6.8) in \cite{gs_Naturality}.


We call a family of functions $\At$ on $\Omega$ with values in $[-\infty,+\infty]$ an \emph{additive functional} (AF in abbreviation), if there exist a \emph{defining set} $\Lambda\in\F_\infty$ and an $\excm$-inessential set $N$ (called an \emph{exceptional set} for $\At$) such that the following conditions hold:
\begin{enumerate}[label=(A.\arabic*)]
    \item $A_t$ is measurable with respect to $\F_t$.
    \item $\Px(\Lambda) = 1$ for all $x\in E\setms N$.
    \item $\theta_t[\Lambda]\subseteq \Lambda$ for all $t>0$.
    \item $A_0(\omega)=0$ for all $\omega\in\Lambda$.
    \item $\abs{A_t(\omega)}<\inft$ for all $\omega\in\Lambda$, $t<\zeta(\omega)$.
    \item $A_t(\omega) = A_{\zeta(\omega)}(\omega)$ for all $\omega\in\Lambda$, $t\ge\zeta(\omega)$.
    \item $A_{s+t}(\omega) = A_s(\omega) + A_t(\theta_s\omega)$ for all $\omega\in\Lambda$, $t,s\ge0$.
    \item \label{ass:AF}
    For each $\omega\in\Lambda$, the function $[0,\zeta(\omega))\ni t\mapsto A_t(\omega)$ is right continuous and has left limits.
\end{enumerate}
We say that $\At$ is a \emph{positive additive functional} (PAF in abbreviation) if,
in addition, for all $\omega\in\Lambda$, $t\geq0$, we have $A_t(\omega)\geq0$.
We say that $\At$ is a \emph{positive continuous additive functional} (PCAF in abbreviation), when we replace condition \ref{ass:AF} in the above definition with an even stronger condition:
\begin{enumerate}
    \item[(A.8')] 
    For each $\omega\in\Lambda$, the function $[0,\inft)\ni t\mapsto A_t(\omega)$ is continuous and non\nobreakdash-negative.
\end{enumerate}

Later, we will see that each PCAF corresponds to a unique Borel measure within the class of so-called smooth measures.

Denote by $\BEstar$ the \emph{universal completion} of $\BE$, that is, $\BEstar := \bigcap_\mu \mathcal{B}^\mu(E)$, where $\mathcal{B}^\mu(E)$ is the completion of $\BE$ with respect to the measure $\mu$ and the above intersection is taken over all probability measures on $\EBE$.
By $\BeE$ we denote the $\sigma$\nobreakdash-field generated by $\BEstar$\nobreakdash-measurable functions which are $\alpha$\nobreakdash-excessive for some $\alpha\ge0$.
A set $B\in\BeE$ is called \emph{$\excm$-semi\nobreakdash-polar}, if
$\Pexcm(X_t\in B\text{ for uncountably many }t)=0$.
Denote by $\tau_B := \inf\{t>0: X_t\notin B\}$ the \emph{exit time} of $\Xt$ from a set $B$.
Note that $\tau_B = \sigma_{E_\Cem\setms B}$.
We say that an increasing sequence of sets $(B_n)_{n\in\N}$ in $\BeE$ is a \emph{$\excm$-generalized nest} provided
$\Pexcm(
        \lim_{n\to\inft} \tau_{B_n} < \zeta
    ) = 0$.
We call a set $B$ \emph{finely open} if for each $x\in B$ there exists a nearly Borel set $B_x\supseteq B^c$ such that
$\Px(\sigma_{B_x}>0)=1$; see Definition (4.1) in \cite{blumenthal_getoor}.
Finally, we call a Borel measure $\nu$ \emph{smooth} if it charges no $\excm$-semi\nobreakdash-polar set and admits an associated $\excm$-generalized nest $(G_n)_{n\in\N}$ of finely open sets with $\nu(G_n)<\inft$ for each $n\in\N$.

When we restrict our considerations to symmetric processes $\Xt$, then there is a regular Dirichlet form $\EDE$ associated with $\Xt$ and the class of smooth measures can be defined more easily by employing the notion of capacity corresponding to $\EDE$.

Let $\mathcal{O}$ be the family of all open subsets of $E$. For $A\subseteq E$ denote by $\mathcal{O}_A$ the family of all open supersets of $A$. For $U\in\mathcal{O}$, we define
$\mathcal{L}_U := \{u\in\DE: u\geq1\ m\text{-a.e. on }U\}$. The \emph{capacity} corresponding the Dirichlet form $\EDE$ is defined in the following way: for $A\in\mathcal{O}$ we define
\begin{align}
\label{eq:cap1}
    \Capa(A)
    :=
    \begin{cases}
        \inf_{u\in\mathcal{L}_U} (\E[u] + \normL{u}{2}^2) & \text{if } \mathcal{L}_U\neq \emptyset , \\
        \inft & \text{if } \mathcal{L}_U = \emptyset ,
    \end{cases}
\end{align}
and for any $A\subseteq E$,
\begin{align}
\label{eq:cap2}
    \Capa(A) := \inf_{U\in\mathcal{O}_A} \Capa(U).
\end{align}
An increasing sequence of sets $(B_n)_{n\in\N}$ is a $m$-generalized nest if and only if
\begin{align*}
    \lim_{n\to\inft} \Capa(K\setms B_n) = 0
    \quad
    \text{for every compact set }K.
\end{align*}
Finally, a Borel measure $\nu$ is smooth if and only if it charges no set of zero capacity and admits an associated $m$-generalized nest $(F_n)_{n\in\N}$ of closed sets such that $\nu(F_n)<\inft$ for each $n\in\N$. For a proof of the equivalence of the two definitions of a smooth measure given above; see \cite{fg_smooth}. See also \cite[p.~83]{fot} or Appendix~B of Li and Ying \cite{BivariateRevuz}.

We are ready to establish the correspondence between the class of PCAFs and the class of smooth measures. Let $\At$ be an additive functional and fix some excessive measure $\excm$ (with respect to $\Xt$).
The \emph{Revuz measure} $\nu$ of $\At$ related to a reference measure $\excm$ is defined by
\begin{align}
\label{eq:RevuzAF}
    \intE f \,\dd \nu
    :=
    \lim_{t\to0^+}
    \frac{1}{t}
    \intE
    \Ev_x \left(
        \ \int\limits_{(0,t]} f(X_s) \,\dd A_s
    \right)
    \,\excm(\dx),
    \quad
    f\in\BeE_+, 
\end{align}
or, equivalently,
\begin{align*}
    \intE f \,\dd \nu
    =
    \lim_{\beta\to \inft}
    \beta
    \intE
    \Ev_x \left(
        \ \int\limits_0^\inft e^{-\beta t} f(X_t) \,\dd A_t
    \right)
    \,\excm(\dx),
    \quad
    f\in\BeE_+.
\end{align*}
Here, $\BeE_+$ is the class of non\nobreakdash-negative $\BeE$\nobreakdash-measurable functions.
We say that such $\At$ and $\nu$ are in the \emph{Revuz correspondence}. The notions of Revuz measure and Revuz correspondence are named in honor of Revuz; see \cite{Revuz}.

We provide the following simple example of $\nu$ and $\At$ in the Revuz correspondence characterized by a Borel function $g$.
Let
\begin{align}
\label{eq:AtInt}
    A_t(\omega)
    :=
    \int\limits_0^t g(X_s(\omega)) \dss,
    \quad
    t\ge0.
\end{align}
Then $\At$ is an AF, PAF, and PCAF for $g$ bounded, non\nobreakdash-negative, and bounded non\nobreakdash-negative, respectively. Moreover
$\At$ is
in the Revuz correspondence (related to the reference measure $m$) with measure $g(x)m(\dx)$. For this we refer to \cite[p.~41]{ffgk}.
We may refer also to Section~3 of \cite{Trutnau}, where a similar result in the context of generalized Dirichlet forms was derived.

We say that processes $\At$ and $\Bt$ are \emph{$\excm$-equivalent}, if
\begin{align}
\label{eq:AFeq}
    \Pexcm(A_t \neq B_t)=0,
    \quad
    \text{for all }t>0.
\end{align}
In such a situation, we write $A_t \equi{\excm} B_t$.

The Revuz correspondence is a one\nobreakdash-to\nobreakdash-one correspondence between the class of smooth measures and the class of PCAFs (up to $\equi{\excm}$-equivalence).
For this statement, we refer to Theorem (3.11) in \cite{fg_smooth}.

If we limit our discussion to the symmetric processes $\Xt$, the Revuz correspondence can be characterized by the following formula:
\begin{align}
\label{eq:RevuzSym}
    \intE f h \,\dd \nu
    &=
    \lim_{t\to0^+}
    \frac{1}{t}
    \intE h(x)
    \Ev_x \left(
        \int\limits_0^t f(X_s) \,\dd A_s
    \right)
    \,\excm(\dx)
    \\ \notag
    &=
    \lim_{\beta\to \inft}
    \beta
    \intE
    h(x)
    \Ev_x \left(
        \int\limits_0^\inft e^{-\beta t} f(X_t) \,\dd A_t
    \right)
    \,\excm(\dx),
\end{align}
where $h$ is any $\alpha$\nobreakdash-excessive function ($\alpha\ge0$) and $f$ is any non\nobreakdash-negative Borel function; see Section~5.1 of~\cite{fot}. See also
Chapter~2 of Fukushima \cite{ffgk}
and
Appendix~B of \cite{BivariateRevuz}.

Several generalizations of identity \eqref{eq:RevuzSym} to the non-symmetric case have been proposed in the literature. In these studies, authors rely on the non-symmetric generalization of Dirichlet form: semi-Dirichlet form. Under some technical assumptions, like the sector condition, the Revuz correspondence can be characterized by:
\begin{align*}
    \intE f \hat{h} \,\dd \nu
    &=
    \lim_{t\to0^+}
    \frac{1}{t}
    \intE \hat{h}(x)
    \Ev_x \left(
        \int\limits_0^t f(X_s) \,\dd A_s
    \right)
    \,\excm(\dx)
    \\ \notag
    &=
    \lim_{\beta\to \inft}
    \beta
    \intE
    \hat{h}(x)
    \Ev_x \left(
        \int\limits_0^\inft e^{-\beta t} f(X_t) \,\dd A_t
    \right)
    \,\excm(\dx),
\end{align*}
where $\hat{h}$ is any quasi-continuous $\alpha$\nobreakdash-coexcessive function ($\alpha\ge0$). We refer here to \cite{BivariateRevuz} and \cite{Fitzsimmons}.
To prove the main result of this paper, the above formula is needed for a special case of the process $X$. To be precise, we require this formula for the space-time process of a symmetric Hunt process. To bypass the mentioned assumptions, we prove the above formula for this special case in Subsection~\ref{sub:space_time} below.


\subsection{Space-time process}
\label{sub:space_time}

Here, we revert to our global assumption that $X=(\Omega, \F, \Ft, (\theta_t), (X_t), (\Prb_x))$ is a symmetric Hunt process.
By $\st{X}=(\st{\Omega}, \st{\F}, (\st{X}_t), (\st{\Prb}_{(t_0,x)}))$ we denote its \emph{space\nobreakdash-time process}. Here,
\begin{align*}
    \st{\Omega} &:= (0,\infty)\tim\Omega, \\
    \st{\F} &:= \sigma\{\mathcal{B}((0,\infty))\tim\F \}, \\
    \st{\Prb}_{(t_0,x)} &:= \delta_{t_0} \otimes \Px,
    \quad (t_0,x)\in\st{E}_\Cem, \\
    \st{X}_t(\tau,\omega) &:=
        \begin{cases}
            (\tau+t,X_t(\omega)) & \text{when } t<\zeta(\omega) , \\
            (\infty,\Delta) & \text{otherwise},
        \end{cases}
\end{align*}
where $\st{E}_\Cem  := (0,\infty) \tim E \cup \{(\infty, \Cem)\}$. Here, $\delta_{t_0}$ denotes the Dirac delta measure at a point $t_0\in(0,\infty)$ and the point $\{(\infty, \Cem)\}$ is the cemetery state of $\Xbrevet$. The space\nobreakdash-time process $\Xbrevet$ is a non\nobreakdash-symmetric Hunt process with respect to the following admissible filtration
\begin{align*}
    \st{\F}_t := \sigma\{\mathcal{B}((0,\infty)) \tim \F_t\}.
\end{align*}

Denote the expected value on the space $(\st{\Omega}, \st{\F}, \st{\Prb}_{(t_0,x)})$ by $\Ev_{(t_0,x)}$.
A conditional expectation of any $\st{\F}_\infty$-measurable $\Phi\colon \st{\Omega}\to\R$ given $\st{\F}_t$ can be written as
\begin{align}
\label{eq:E-sp-tim}
    \Ev_{(t_0,x)} [\Phi|\st{\F}_t]
    =
    \Ev_x[\Phi(t_0,\cdot)|\F_t],
    \quad
    \st{\Prb}_{(t_0,x)}\text{-a.s.}
\end{align}
Indeed, for any $A\in\mathcal{B}((0,\infty))$ and $B\in\F_t$
\begin{align*}
    \Ev_{(t_0,x)} (\Phi \ind_{A\tim B})
    &=
    \ind_A(t_0) \Ev_x( \Ev_x[\Phi(t_0,\cdot)|\F_t] \ind_B)
    =
    \Ev_{(t_0,x)} ( \Ev_x[\Phi(t_0,\cdot)|\F_t] \ind_{A\tim B}).
\end{align*}

We provide a brief overview of the essential properties of the space-time process, primarily following \cite[Section~16]{Sharpe}.
We denote the Lebesgue measure on $(0,\infty)$ by $\dr$ and write the space $L^2((0,\infty)\times E,\mathcal{B}((0,\infty)\times E),\dr \otimes m)$  shortly as $L^2(\dr \otimes m)$.
Let $\Pbrevet$ be the semigroup of the space\nobreakdash-time process $\Xbrevet$. Then $\st{P}_t$ can be written as
\begin{align*}
    \st{P}_t f(r,x) = P_t[f(r+t,\cdot)](x),
    \quad
    f\in L^2(\dr \otimes m).
\end{align*}
Its adjoint operator $\bst{P}_t$ on $L^2(\dr \otimes m)$ is given by
\begin{align*}
    \bst{P}_t g (r,x) = \ind_{[0,r)}(t) P_t[g(r-t,\cdot)](x),
    \quad
    g\in L^2(\dr \otimes m).
\end{align*}
At the same time, the semigroup $\Pbstt$ is associated with the \emph{backward space\nobreakdash-time process} $\Xbstt$ defined on the same family of probability spaces
$(\st{\Omega}, \st{\F}, (\st{\Prb}_{(t_0,x)}))$ by
\begin{align*}
    \bst{X}_t(\tau,\omega) :=
    \begin{cases}
        (\tau-t,X_t(\omega)) & \text{for } t<\tau \wedge \zeta(\omega) , \\
        (\infty, \Cem) & \text{otherwise}.
    \end{cases}
\end{align*}
To rephrase, $\Xbstt$ represents the Cartesian product of $\Xt$ and the uniform motion to the left on $(0,\infty)$ with killing at $0$.

Note that the measure $\dr \otimes m$ is excessive with respect to the semigroup $\Pbrevet$. Indeed, for $\st{B}\in\mathcal{B}((0,\inft)\times E)$, by the symmetry of $\Pt$, we can write
\begin{align*}
    \intb_0^\inft \intE
        \st{P}_t \ind_{\st{B}}
    \dms \dr
    &=
    \intb_0^\inft \intE
        P_t[\ind_{\st{B}} (r+t,\cdot)](x)
    \,m(\dx)\dr
    =
    \intb_0^\inft \intE
       \ind_{\st{B}} (r+t,x) P_t1(x)
    \,m(\dx)\dr
    \\
    &\leq
    \intb_0^\inft \intE
       \ind_{\st{B}} (r+t,x)
    \,m(\dx)\dr
    =
    \intb_t^\inft \intE
       \ind_{\st{B}} (r,x)
    \,m(\dx)\dr
    \leq
    \dr \otimes m (\st{B}).
\end{align*}

At the end of this subsection,
we prove a non\nobreakdash-symmetric variant of equation \eqref{eq:RevuzSym} for the space\nobreakdash-time process. This result will be useful in 
Section~\ref{sec:Sharp_bracket}.
\begin{lem}
\label{lem:RevuzST}
    Let $\At$ be a PCAF related to the space\nobreakdash-time process $\Xbrevet$ described above. Let $\nu$ be the Revuz measure of $\At$ (related to the reference measure $\dr \otimes m$). Then $\nu$ satisfies
    \begin{align}
    \label{eq:RevuzST}
        \int\limits_{\mathclap{(0,\inft)\tim E}} f \hat{h} \,\dd \nu
        &=
        \lim_{t\to0^+}
        \frac{1}{t}
        \int\limits_{\mathclap{(0,\inft)\tim E}} \hat{h}(r,x)
        \Ev_{(r,x)} \left(
            \int\limits_0^t f(\st{X}_s) \,\dd A_s
        \right)
        \, \dr m(\dx)
    \end{align}
    for any $f\in\mathcal{B}^e((0,\infty)\tim E)_+$ and $\alpha$\nobreakdash-coexcessive function $\hat{h}$ ($\alpha\ge0$).
\end{lem}

Before we derive the proof of this lemma, we need to introduce the notion of the $\alpha$\nobreakdash-subprocess following \cite[Exercise (12.34)]{Sharpe}. Fix $\alpha>0$. We want to define the process $\subpr{\st{X}}$, which is obtained by killing $\Xbrevet$ at a random and independent time $T^\alpha$ following the exponential distribution with parameter $\alpha$. To be precise, let
\begin{align*}
    \subpr{\st{\Omega}} &:= (0,\inft)\tim \st{\Omega}, \\
    \subpr{\st{\F}} &:= \sigma\{\mathcal{B}((0,\infty))\tim\st{\F} \}, \\
    \subpr{\lambda}(\dd\xi) &:=  \alpha e^{-\alpha\xi} \dd \xi, \\
    \subpr{\st{\Prb}}_{(t_0,x)} &:= \subpr{\lambda} \otimes \st{\Prb}_{(t_0,x)},
    \quad (t_0,x)\in\st{E}_\Cem, \\
    \st{X}_t(\xi,\tau,\omega) &:=  \st{X}_t(\tau,\omega),\\
    \subpr{\st{X}}_t(\xi,\tau,\omega) &:=
        \begin{cases}
            \st{X}_t(\xi,\tau,\omega) & \text{when } t<\xi , \\
            (\infty,\Delta) & \text{otherwise}.
        \end{cases}
\end{align*}
Let
$T^\alpha(\xi,\tau,\omega):=\xi$ and
$\subpr{\zeta}:=\zeta \wedge T^\alpha$
be the lifetime of $\subpr{\st{X}}$.
We call $\subpr{\st{X}}$ the \emph{$\alpha$-subprocess} of $\Xbrevet$.
The semigroup of the process $\subpr{\st{X}}$ is $(e^{-\alpha t}\st{P}_t)$. We use the notation
$\Ev_{(t_0,x)}$
either for the expected value on the spaces
$(\st{\Omega}, \st{\F}, \st{\Prb}_{(t_0,x)})$
or
$(\subpr{\st{\Omega}}, \subpr{\st{\F}}, \subpr{\st{\Prb}}_{(t_0,x)})$.
Similarly, we define the $\alpha$-subprocess of $\bst{X}$ denoted by $\bst{X}^\alpha$.

\begin{proof}[Proof of Lemma~\ref{lem:RevuzST}]
    Recall that the semigroup of $\st{X}^\alpha$ is $(e^{-\alpha t}\st{P}_t)$.
    Consequently, as $\hat{h}$ is $\alpha$\nobreakdash-coexcessive for $\Xbrevet$, the measure $\hat{h} m$ is excessive for $\st{X}^\alpha$.
    Additionally, the mapping
    $[0,\inft)\ni t\mapsto \hat{h}(\bst{X}^\alpha_t)$ is $\subpr{\Prb}_m$\nobreakdash-almost surely right continuous; see Proposition (4.2) and Theorem (4.8) in \cite{blumenthal_getoor}. Due to self\nobreakdash-duality of $\Xt$ (see \eqref{eq:self-dual}), the mapping
    $(0,\zeta^\alpha)\ni t\mapsto \hat{h}(\st{X}^\alpha_{t-})$ is $\subpr{\Prb}_m$\nobreakdash-almost surely left continuous.

    Let $A^\alpha_t$ denote the PCAF $\At$ stopped at time $T^\alpha$.
    As a consequence of the above observations, we can invoke Corollary~(8.11) of \cite{gs_Naturality} and write
    \begin{align*}
        \lim_{t\to0^+}
        \frac{1}{t}
        \int\limits_{\mathclap{(0,\inft)\tim E}}
        &\Ev_{(r,x)} \left(
            \int\limits_0^t f(\st{X}^\alpha_{s-}) \hat{h}(\st{X}^\alpha_{s-}) \,\dd A^\alpha_s
        \right)
        \,\dr m(\dx)
        \\
        &=
        \lim_{t\to0^+}
        \frac{1}{t}
        \int\limits_{\mathclap{(0,\inft)\tim E}} \hat{h}(r,x)
        \Ev_{(r,x)} \left(
            \int\limits_0^t f(\st{X}^\alpha_{s-}) \,\dd A^\alpha_s
        \right)
        \,\dr m(\dx).
    \end{align*}

    Given that $\At$ is continuous and the discontinuity points of the sample paths $t\mapsto \st{X}^\alpha_t$ are at most countable, we may reformulate the above equation as follows:
    \begin{align*}
        \lim_{t\to0^+}
        \frac{1}{t}
        \int\limits_{\mathclap{(0,\inft)\tim E}}
        &\Ev_{(r,x)} \left(
            \int\limits_0^t f(\st{X}^\alpha_s) \hat{h}(\st{X}^\alpha_s) \,\dd A^\alpha_s
        \right)
        \,\dr m(\dx)
        \\
        &=
        \lim_{t\to0^+}
        \frac{1}{t}
        \int\limits_{\mathclap{(0,\inft)\tim E}} \hat{h}(r,x)
        \Ev_{(r,x)} \left(
            \int\limits_0^t f(\st{X}^\alpha_s) \,\dd A^\alpha_s
        \right)
        \,\dr m(\dx).
    \end{align*}
    We clam that we can replace $\st{X}^\alpha_s$ and $A^\alpha_s$ by $\st{X}_s$ and $A_s$ in the above formula.
    Denote
    $B_t := \int_0^t f(\st{X}_s) \,\dd A_s$
    and 
    $B_t^\alpha := \int_0^t f(\st{X}^\alpha_s) \,\dd A^\alpha_s
    = B_{t\wedge T^\alpha}$
    .
    We can write
    \begin{align*}
        &\lim_{t\to0^+}
        \frac{1}{t}
        \int\limits_{\mathclap{(0,\inft)\tim E}} \hat{h}(r,x)
        \Ev_{(r,x)} \left(
            \int\limits_0^t f(\st{X}^\alpha_s) \,\dd A^\alpha_s
        \right)
        \,\dr m(\dx)
        \\
        &=
        \lim_{t\to0^+}
        \left[ \,
            \frac{1}{t}
            \int\limits_{\mathclap{(0,\inft)\tim E}} \hat{h}(r,x)
            \Ev_{(r,x)} B_t^\alpha \ind_{\{T^\alpha > t\}}
            \,\dr m(\dx)
            +
            \frac{1}{t}
            \int\limits_{\mathclap{(0,\inft)\tim E}} \hat{h}(r,x)
            \Ev_{(r,x)} B_t^\alpha \ind_{\{T^\alpha \leq t\}}
            \,\dr m(\dx)
        \right]
        .
    \end{align*}
    When $\st{\omega}\in\subpr{\st{\Omega}}$ is such that $T^\alpha(\st{\omega}) > t$, then $B_t^\alpha (\st{\omega}) = B_t (\st{\omega})$. Moreover, as $B_t$ and $T^\alpha$ are independent, we can write
    $\Ev_{(r,x)} B_t^\alpha \ind_{\{T^\alpha > t\}}
    =
    \Ev_{(r,x)} B_t \cdot \Ev_{(r,x)} \ind_{\{T^\alpha > t\}}
    =
    e^{-\alpha t}\Ev_{(r,x)} B_t$. Therefore
    \begin{align}
    \label{eq:RevuzST_eq1}
        \lim_{t\to0^+}
        \frac{1}{t}
        \int\limits_{\mathclap{(0,\inft)\tim E}} \hat{h}(r,x)
        \Ev_{(r,x)} B_t^\alpha \ind_{\{T^\alpha > t\}}
        \,\dr m(\dx)
        =
        \lim_{t\to0^+}
        \frac{1}{t}
        \int\limits_{\mathclap{(0,\inft)\tim E}} \hat{h}(r,x)
        \Ev_{(r,x)} B_t
        \,\dr m(\dx).
    \end{align}
    When the right\nobreakdash-hand side of \eqref{eq:RevuzST_eq1} is infinite, then the claim is obviously true. Assume that it is finite. Then, since $B_t^\alpha \leq B_t$, we get
    $
    \Ev_{(r,x)} B_t^\alpha \ind_{\{T^\alpha \leq t\}}
    \leq
    \Ev_{(r,x)} B_t \ind_{\{T^\alpha \leq t\}}
    =
    \Ev_{(r,x)} B_t \cdot \Ev_{(r,x)} \ind_{\{T^\alpha \leq t\}}
    =
    (1-e^{-\alpha t}) \Ev_{(r,x)} B_t
    $.
    Therefore,
    \begin{align*}
        \lim_{t\to0^+}
        \frac{1}{t}
        \int\limits_{\mathclap{(0,\inft)\tim E}} \hat{h}(r,x)
        &\Ev_{(r,x)} B_t^\alpha \ind_{\{T^\alpha \leq t\}}
        \,\dr m(\dx)
        \\
        &\leq
        \lim_{t\to0^+}
        (1-e^{-\alpha t})
        \lim_{t\to0^+}
        \frac{1}{t}
        \int\limits_{\mathclap{(0,\inft)\tim E}} \hat{h}(r,x)
        \Ev_{(r,x)} B_t
        \,\dr m(\dx)
        =
        0.
    \end{align*}
    Summarizing, we show that 
    \begin{align*}
        \lim_{t\to0^+}
        \frac{1}{t}
        \int\limits_{\mathclap{(0,\inft)\tim E}} \hat{h}(r,x)
        &\Ev_{(r,x)} \left(
            \int\limits_0^t f(\st{X}^\alpha_s) \,\dd A^\alpha_s
        \right)
        \,\dr m(\dx)
        \\
        &=
        \lim_{t\to0^+}
        \frac{1}{t}
        \int\limits_{\mathclap{(0,\inft)\tim E}} \hat{h}(r,x)
        \Ev_{(r,x)} \left(
            \int\limits_0^t f(\st{X}_s) \,\dd A_s
        \right)
        \,\dr m(\dx)
        .
    \end{align*}
    Analogously, one can show that
    \begin{align*}
        \lim_{t\to0^+}
        \frac{1}{t}
        \int\limits_{\mathclap{(0,\inft)\tim E}}
        &\Ev_{(r,x)} \left(
            \int\limits_0^t f(\st{X}^\alpha_s) \hat{h}(\st{X}^\alpha_s) \,\dd A^\alpha_s
        \right)
        \,\dr m(\dx)
        \\
        &=
        \lim_{t\to0^+}
        \frac{1}{t}
        \int\limits_{\mathclap{(0,\inft)\tim E}}
        \Ev_{(r,x)} \left(
            \int\limits_0^t f(\st{X}_s) \hat{h}(\st{X}_s) \,\dd A_s
        \right)
        \,\dr m(\dx)
        .
    \end{align*}
    
    To summarize,
    \begin{align*}
        \lim_{t\to0^+}
        \frac{1}{t}
        \int\limits_{\mathclap{(0,\inft)\tim E}}
        &\Ev_{(r,x)} \left(
            \int\limits_0^t f(\st{X}_s) \hat{h}(\st{X}_s) \,\dd A_s
        \right)
        \,\dr m(\dx)
        \\
        &=
        \lim_{t\to0^+}
        \frac{1}{t}
        \int\limits_{\mathclap{(0,\inft)\tim E}} \hat{h}(r,x)
        \Ev_{(r,x)} \left(
            \int\limits_0^t f(\st{X}_s) \,\dd A_s
        \right)
        \,\dr m(\dx).
    \end{align*}
    Finally, since $\nu$ is the Revuz measure of $\At$, it follows from \eqref{eq:RevuzAF} that the left\nobreakdash-hand side is equal to $\int_{(0,\inft)\tim E} f \hat{h} \,\dd \nu$.
\end{proof}

\subsection{Martingale}
\label{sub:martingale}

For any $x\in E_\Cem$ a real valued process $((M_t),\Ft,\Px)$ is a \emph{martingale} provided $\Ev_x[M_t|\F_s]=M_s$ $\Px$-a.s. for all $s\leq t$. We say that $M=((M_t),\Ft,\Pxx)$ is a \emph{martingale} when $((M_t),\Ft,\Px)$ is a martingale for all $x\in E_\Cem$.

Let $M=((M_t),\Ft,\Pxx)$ be a martingale such that
$M_0=0$
and
$\Ev_x M_t^2 < \infty$ for $t>0, x\in E_\Cem$.
There exists
a unique integrable and predictable increasing process
$\langle M \rangle$ starting at zero such that $((M_t^2 - \langle M \rangle_t),\Ft,\Pxx)$ is a martingale.
In particular,
\begin{align}
\label{eq:sharp_square}
    \Ev_x \langle M \rangle_t = \Ev_x M_t^2,
    \quad
    \text{for all }
    t\geq0, x\in E_\Cem.
\end{align}
In addition, whenever $\Mt$ is an AF, $\langle M \rangle$ is
a unique PCAF (up to $\equi{m}$-equivalence) such that $((M_t^2 - \langle M \rangle_t),\Ft,\Pxx)$ is a martingale.
We call the process $\langle M \rangle$ the \emph{sharp bracket} (or \emph{predictable quadratic variation}) of $\Mt$.
For more details, we refer to Section~A.3 of \cite{fot}. See also Section~4 of \cite{Trutnau} and Th\'eor\`eme~3 in Meyer \cite[III]{Meyer}.

The \emph{square bracket} of the process $\Mt$, denoted by $[M]$, may be defined by
\begin{align*}
    [M]_t := \langle M^c \rangle_t + \sum_{s\le t} \Delta M^2_s,
\end{align*}
where $\Delta X_t := X_t - X_{t-}$ and $M^c_t$ is the continuous part of $\Mt$; see (A.3.6) in \cite{fot}.
Observe that $[M]_t=\langle M \rangle_t$ for continuous $\Mt$.

Let $1\le p<\infty$. The Burkholder--Davies--Gundy inequality states that:
\begin{align}
\label{sim:bdg}
 \Ev_x |M_T|^p \asymp \Ev_x([M]_T)^{p/2},
 \quad
 T>0, x\in E_\Cem.
\end{align}
It should be noted that, in general, the square bracket $[M]$ cannot be replaced by the sharp bracket $\langle M \rangle$. However, as demonstrated in Lenglart, L\'{e}pingle, and Pratelli \cite[Remarque~4.2]{llp} and Barlow, Jacka, and Yor \cite[(4.b’), Table~4.1, p.~162]{bjy}, the following one-sided estimates hold: if $2\le p\le \infty$, then
\begin{align}
\label{eq:brackets-}
 \Ev_x (\langle M \rangle_t)^{p/2} & \lesssim \Ev_x ([M]_t)^{p/2},
 \quad
 t>0, x\in E_\Cem,
\end{align}
and, when $1\le p \le 2$,
\begin{align}\label{eq:brackets+}
 \Ev_x ([M]_t)^{p/2} & \lesssim \Ev_x (\langle M \rangle_t)^{p/2}
 \quad
 t>0, x\in E_\Cem.
\end{align}

\section{Sharp bracket of parabolic martingale}
\label{sec:Sharp_bracket}

Fix $T > 0$ and $f\in\LpE$.
In the present article, we consider the so\nobreakdash-called \emph{parabolic martingale}:
\begin{align}
\label{eq:Mt}
    M_t  := P_{(T-t)\vee 0}f(X_{t\wedge T}) - P_T f(X_0).
\end{align}
The process $\Mt$ is a martingale $P_{T-t}f(X_t) - P_T f(X_0)$ starting at zero and stopped at $T$ with
$M_t = f(X_T) - P_T f(X_0)$ for $t\ge T$.

The parabolic martingale was first used in Ba{\~n}uelos and M{\'e}ndez-Hern{\'a}ndez~\cite{bm} for Brownian motion to obtain $L^p$-estimates for the {Beurling}--{Ahlfors} operator.
Further applications may be found in Applebaum and Ba{\~n}uelos~\cite{ab} and Ba{\~n}uelos and Bogdan~\cite{bb} where more general L{\'e}vy processes are considered.
See also \cite{bkp}, where the parabolic martingale was used to derive the Hardy--Stein identity in Orlicz spaces.

The main result of this section is the following relation between the sharp bracket of the martingale $\Mt$ and the carr\'{e} du champ operator $\Gamma$:
\begin{align}
    \label{eq:sharpMt}
    \langle M \rangle_t
    \equi{m}
    \int\limits_0^{t\wedge T} 2\Gamma [P_{T-s}f](X_s) \dss
    =
    \int\limits_0^{t\wedge T} \intEXs (P_{T-s}f(y) - P_{T-s}f(X_s))^2 \,J(X_s,\dy)\ds
    ,
\end{align}
when the function $f$ belongs to $\DE$; see Corollary~\ref{cor:sharpMt} below. Here, $\equi{\excm}$ is the equivalence relation defined by \eqref{eq:AFeq} and $\Gamma$ is given by \eqref{eq:carre}.
This result generalizes Lemma~3.1 of \cite{lw}, which requires some continuity assumption on the semigroup $\Pt$.
While the main idea of the proof remains the same, omitting the mentioned assumption requires employing the Revuz correspondence for certain functionals.

\begin{rem}
    We would like to clarify certain inaccuracies that have appeared in the literature in the context of identity \eqref{eq:sharpMt}.
    Since the process is pure-jump, the martingale $\Mt$ is not continuous in general and $[M]\neq\langle M \rangle$, hence the same statement does not hold for the square bracket $[M]$.
    While the sharp bracket $\langle M \rangle$ admits a continuous version, the same assertion cannot be assumed for $\Mt$, as was done in \cite[Proof of Lemma~3.1]{lw}. Due to this error, the authors conclude \eqref{eq:sharpMt} with $\langle M \rangle$ replaced by $[M]$; see \cite[(3.2)]{lw}.
    
    The same incorrect equation appears in \cite{bbl} for $\Xt$, assuming it is a pure-jump L{\'e}vy process; see the first displayed equation on page 471 therein. The authors refer to \cite{ab}, where the following identity is displayed; see \cite[(3.4)]{ab}. Rewriting it to our context,
    \begin{align*}
        [ M ]_t
        =
        \int\limits_0^{t\wedge T} \intRn (P_{T-s}f(X_{s-}+y) - P_{T-s}f(X_{s-}))^2 \,\tilde{N}(\ds,\dy)
        ,
    \end{align*}
    where $\tilde{N}$ is a compensated Poisson random measure of $\Xt$. This identity is an application of It\^o's formula. The authors of \cite{bbl} incorrectly rewrite the above equation in the form of \eqref{eq:sharpMt} with $\langle M \rangle$ replaced by $[M]$.
    
    Nevertheless, we want to stress that the mentioned inaccuracies do not affect the correctness of the main results of \cite{bbl} and \cite{lw}. Due to estimates \eqref{eq:brackets-} and \eqref{eq:brackets+}, the reasoning in \cite{bbl} can be easily corrected. The main result of \cite{lw}, however, remains irreparable due to a different issue; see Remark~\ref{rem:GlessH} below.
\end{rem}

Let us return to the discussion of the proof of the relation \eqref{eq:sharpMt}.
Since the martingale $\Mt$ is not an additive functional, we cannot apply the Revuz correspondence directly. Nevertheless, its time\nobreakdash-space counterpart is an additive functional:
\begin{align}
\label{eq:Mt-sp-tim}
    \st{M}_t(\tau,\omega)
    :=
    \begin{cases}
        P_{T-\tau-t} f(X_t(\omega)) - P_{T-\tau}f(X_0(\omega)) & \text{if } t+\tau\le T, \\
        f(X_{T-\tau}(\omega)) - P_{T-\tau}f(X_0(\omega)) & \text{if } \tau\le T <t+\tau, \\
        0 & \text{if } T <\tau.
    \end{cases}
\end{align}
Writing compactly,
\begin{align*}
    \st{M}_t(\tau,\omega)
    =
    P_{(T-\tau-t)\vee 0} f(X_{[t\wedge(T-\tau)]\vee 0}(\omega))
    -
    P_{(T-\tau)\vee 0} f(X_0(\omega)).
\end{align*}
Note that $\langle \st{M} \rangle$ is a PCAF.
Observe that $\st{M}_t(0,\cdot) = M_t$ and $\st{M}_0=0$.

We begin by showing that $\st{M}$ is indeed a martingale. 
\begin{prop}
    The process $\st{M}=((\st{M}_t),(\st{\F}_t),(\st{\Prb}_{(t_0,x)})$ given by \eqref{eq:Mt-sp-tim} is a martingale.
\end{prop}
\begin{proof}
    According to the fact that $\st{M}$ is the process $P_{T-\tau-t} f(X_t) - P_{T-\tau}f(X_0)$ stopped at $t=T-\tau$, it is enough to consider $t+\tau\le T$.
    Employing the Markov property of the Hunt process $\Xt$, we have
    \begin{align}
    \label{eq:MarkovProp-Pt}
        P_{T-t_0-t} f(X_t) = \Ev_{X_t} f(X_{T-t_0-t})
        =
        \Ev_x [f(X_{T-t_0})|\F_t],
        \quad
        \Px\text{-a.s.}
    \end{align}
    for all $x\in E_\Cem$. Let $0\le s\le t\le T-t_0$. Due to \eqref{eq:E-sp-tim} and \eqref{eq:MarkovProp-Pt},
    \begin{align*}
        \Ev_{(t_0,x)} [\st{M}_t | \st{\F}_s]
        &=
        \Ev_x [ P_{T-t_0-t} f(X_t)| \F_s] - P_{T-t_0}f(x)
        \\
        &=
        \Ev_x [ \Ev_x [f(X_{T-t_0})|\F_t]| \F_s] - P_{T-t_0}f(x)
        \\
        &=
        \Ev_x [ f(X_{T-t_0})| \F_s] - P_{T-t_0}f(x),
        \quad
        \Prb_{(t_0,x)}\text{-a.s.}
    \end{align*}
    Again, by \eqref{eq:MarkovProp-Pt},
    \begin{align*}
        \Ev_x [ f(X_{T-t_0})| \F_s] - P_{T-t_0}f(x)
        =
        P_{T-t_0-s} f(X_s) -  P_{T-t_0}f(x)
        =
        \st{M}_s,
        \quad
        \Prb_{(t_0,x)}\text{-a.s.}
    \end{align*}
    Since $x\in E_\Cem$ and $t_0>0$ were chosen arbitrarily, the proof is complete.
\end{proof}

We are ready to prove the main theorem of this section.

\begin{thm}
\label{thm:sharpMt}
    Let $f\in\DE$. Let $\st{M}$ be the martingale given by \eqref{eq:Mt-sp-tim}. Then
    \begin{align}
    \label{eq:At}
        \langle \st{M} \rangle_t
        \equi{\dr \otimes m}
        \int\limits_0^t 2\Gamma [P_{T-\tau-s}f](X_s(\omega)) \ind_{(0,T]}(\tau+s) \dss,
    \end{align}
    where $\Gamma$ is the carr\'{e} du champ operator given by \eqref{eq:carre} and $\equi{\excm}$ is the equivalence relation defined by \eqref{eq:AFeq}.
\end{thm}
\begin{proof}
    Observe that the additive functional on the right-hand side of \eqref{eq:At} takes the form \eqref{eq:AtInt} with $\st{X}_s$ and
    $g(t,x):=2\Gamma [P_{T-t}f](x) \ind_{(0,T]}(t)$. Therefore, its Revuz measure is
    \begin{align*}
        2\Gamma [P_{T-t}f](x) \ind_{(0,T]}(t) m(\dx)\dt.
    \end{align*}

    We denote by $\nu$ the Revuz measure of $\langle\st{M}\rangle$ and take an arbitrary bounded $\alpha$\nobreakdash-coexcessive function $\hat{f}$ with respect to the space-time process $\Xbrevet$.
    By applying \eqref{eq:RevuzST} and \eqref{eq:sharp_square}, we derive
    \begin{align}
    \label{eq:sharpMt-proof-0}
        \int\limits_{\mathclap{(0,\inft)\tim E}} \hat{f} \,\dd \nu
        &=
        \lim_{h\to0^+}
        \frac{1}{h}
        \int\limits_0^\inft \intE \hat{f}(t,x)
        \Ev_{(t,x)}
            \langle\st{M}\rangle_h
        \,m(\dx)\dt
        \\ \notag
        &=
        \lim_{h\to0^+}
        \frac{1}{h}
        \int\limits_0^\inft \intE \hat{f}(t,x)
        \Ev_{(t,x)}
            \st{M}_h^2
        \,m(\dx)\dt.
    \end{align}
    
    If $h+t\le T$, then
    $\Ev_x P_{T-t-h}f(X_h)=P_hP_{T-t-h}f(x) = P_{T-t}f(x)$. Hence, by \eqref{eq:E-sp-tim},
    \begin{align}
    \label{eq:sharpMt-proof-1}
        \Ev_{(t,x)} \st{M}_h^2
        &=
        \Ev_x \st{M}_h^2(t,\cdot)
        =
        \Ev_x (P_{T-t-h}f(X_h))^2 - 2 P_{T-t}f(x) \Ev_x P_{T-t-h}f(X_h)
        +
        (P_{T-t}f(x))^2
        \notag \\
        &=
        P_h [(P_{T-t-h}f)^2](x)
        -
        (P_{T-t}f(x))^2.
    \end{align}
    Moreover, by Lemma~\ref{lem:difPtPTmtf2},
    we get
    \begin{align}
    \label{eq:sharpMt-proof-2}
        P_h [(P_{T-t-h}f)^2]
        -
        (P_{T-t}f)^2
        &=
        \int\limits_0^h 2P_s(\Gamma[P_{T-t-s}f]) \dss
        \\ \notag
        &=
        h \int\limits_0^1 2P_{hs}(\Gamma[P_{T-t-hs}f]) \dss.
    \end{align}
    Here, the integral on the right\nobreakdash-hand side is understood in the Bochner sense in $\LE{1}$.

    Similarly, when $t\le T <h+t$,
    \begin{align}
    \label{eq:sharpMt-proof-1.2}
        \Ev_{(t,x)} \st{M}_h^2
        &=
        P_{T-t} [f^2](x)
        -
        (P_{T-t}f(x))^2
    \end{align}
    and
    \begin{align}
    \label{eq:sharpMt-proof-2.2}
        P_{T-t} [f^2]
        -
        (P_{T-t}f)^2
        &=
        \int\limits_0^{T-t} 2P_s(\Gamma[P_{T-t-s}f]) \dss
        \\ \notag
        &=
        h \int\limits_0^1 2P_{hs}(\Gamma[P_{T-t-hs}f]) \ind_{[0,T-t]}(hs) \dss.
    \end{align}

    Obviously, $\Ev_{(t,x)} \st{M}_h^2=0$ for $T <t$.

    Combining \eqref{eq:sharpMt-proof-0}, \eqref{eq:sharpMt-proof-1}, \eqref{eq:sharpMt-proof-2}, \eqref{eq:sharpMt-proof-1.2}, and \eqref{eq:sharpMt-proof-2.2}, we can write
    \begin{align*}
        \int\limits_{\mathclap{(0,\inft)\tim E}} \hat{f} \,\dd \nu
        &=
        \lim_{h\to0^+}
        \int\limits_0^T \intE \hat{f}(t,\cdot)
        \int\limits_0^1
        2P_{hs}(\Gamma[P_{T-t-hs}f])
        \times \\
        &\quad\times
        \left[
            \ind_{(0,T-h]}(t) + \ind_{[0,T-t]}(hs)\ind_{(T-h,T]}(t)
        \right]
        \,\ds\dm\dt
        \\
        &=
        \lim_{h\to0^+}
        \int\limits_0^T
        \int\limits_0^1
        \intE 2 \hat{f}(t,\cdot)
        P_{hs}(\Gamma[P_{T-t-hs}f])
        \times \\
        &\quad\times
        \left[
            \ind_{(0,T-h]}(t) + \ind_{[0,T-t]}(hs)\ind_{(T-h,T]}(t)
        \right]
        \dms\ds\dt.
    \end{align*}
    The above follows from an application of Tonelli's theorem.
    Note that if $t\in(0,T-h]$, then $t\leq T-h \leq T-hs$, hence $\ind_{[0,T-t]}(hs)=1$.
    Therefore, for all $t\in(0,T]$ we have
    $\ind_{(0,T-h]}(t) + \ind_{[0,T-t]}(hs)\ind_{(T-h,T]}(t) = \ind_{[0,T-t]}(hs)$.
    Now, we claim that
    \begin{align*}
        \lim_{h\to0^+}
        \int\limits_0^T
        \int\limits_0^1
        &\intE 2 \hat{f}(t,\cdot)
        P_{hs}(\Gamma[P_{T-t-hs}f])
        \ind_{[0,T-t]}(hs)
        \dms\ds\dt
        \\
        &=
        \int\limits_0^T
        \int\limits_0^1
        \lim_{h\to0^+}
        \intE 2 \hat{f}(t,\cdot)
        P_{hs}(\Gamma[P_{T-t-hs}f])
        \dms
        \ind_{[0,T-t]}(hs)
        \dss\dt.
    \end{align*}
    For all sufficiently small $h\ge0$, due to the boundedness of $\hat{f}$,
    \begin{align*}
        \intE 2 \hat{f}(t,\cdot)
        P_{hs}(\Gamma[P_{T-t-hs}f])
        \dms
        &\le
        2
        \intE
        \Gamma[P_{T-t-hs}f]
        \dms
        \sup_{(r,x)\in(0,T]\tim E} \abs{\hat{f}(r,x)}
        \\
        &=
        2
        \E[
        P_{T-t-hs}f
        ]
        \sup_{(r,x)\in(0,T]\tim E} \abs{\hat{f}(r,x)}
        \\
        &\le
        2
        \E[f]
        \sup_{(r,x)\in(0,T]\tim E} \abs{\hat{f}(r,x)}
        <
        \inft.
    \end{align*}
    At this point, we made use of the assumption that $f\in\DE$. The dominated convergence theorem implies our claim.

    Lemma~\ref{lem:difPtPTmtf2} states that the mapping $[0,T-t)\ni r\mapsto P_{r}(\Gamma[P_{T-t-r}f])\in\LE{1}$ is continuous. Thus, $P_{hs}(\Gamma[P_{T-t-hs}f])\to\Gamma[P_{T-t}f]$ in $\LE{1}$ as $h\to0^+$. Hence,
    \begin{align*}
        \lim_{h\to0^+}
        \intE 2 \hat{f}(t,\cdot)
        P_{hs}(\Gamma[P_{T-t-hs}f])
        \dms
        =
        \intE 2 \hat{f}(t,\cdot)
        \Gamma[P_{T-t}f]
        \dms.
    \end{align*}
    Clearly,
    $
        \ind_{[0,T-t]}(hs)
        \to
        1
    $
    for every $t\in(0,T]$ as $h\to0^+$.

    Following the above observations, we obtain
    \begin{align*}
        \int\limits_{\mathclap{(0,\inft)\tim E}} \hat{f} \,\dd \nu
        &=
        \int\limits_0^T
        \int\limits_0^1
        \intE 2 \hat{f}(t,\cdot)
        \Gamma[P_{T-t}f]
        \dms
        \ds\dt
        \\
        &=
        \int\limits_0^\inft
        \intE
        \hat{f}(t,x)\cdot 2\Gamma[P_{T-t}f](x) \ind_{(0,T]}(t) \,m(\dx)\dt.
    \end{align*}
    
    The above identity can be extended to any non\nobreakdash-negative $\hat{f}\in\C_c((0,\inft)\tim E)$. Indeed, $\hat{f}$ can be approximated by the bounded $n$\nobreakdash-coexcessive functions
    $\hat{f}_n:=n\hat{R}_n \hat{f}$, where
    $\hat{R}_n$ is the coresolvent of the space-time process $(\st{X}_t)_{t\ge0}$.
    Since $\hat{f}_n\to\hat{f}$ in $\LE{2}$ as $n\to\inft$,
    there exists a subsequence $(\hat{f}_{n_k})_{k\in\N}$ such that $\hat{f}_{n_k}\to\hat{f}$ almost everywhere as $k\to\inft$.

    Finally, since $\hat{f}$ is arbitrary non\nobreakdash-negative function in $\C_c((0,\inft)\tim E)$, we obtain the following equality of the Revuz measures:
    \begin{align*}
        \nu(\dt,\dx) = 2\Gamma [P_{T-t}f](x) \ind_{(0,T]}(t) m(\dx)\dt
        .
    \end{align*}
    The statement of the theorem follows from the uniqueness of the PCAF with a given Revuz measure.
\end{proof}

With the above theorem at hand, due to the fact that $\langle \st{M} \rangle_t(0,\cdot) = \langle M\rangle_t$, we obtain the desired result.

\begin{cor}
\label{cor:sharpMt}
    Let $f\in\DE$. Let $\Mt$ be the martingale given by \eqref{eq:Mt}. Then the equivalence~\eqref{eq:sharpMt} holds.
\end{cor}

\section{Square functions}
\label{sec:square_fun}

\subsection{Square function $G$}

Our starting point is the following square function:
\begin{align*}
    G(x) =
    \left(
        \int\limits_0^\inft \Gamma [P_tf](x) \dts
    \right)^{1/2}.
\end{align*}
It does not possess the desired boundedness of $p$\nobreakdash-norm for all $1<p<\infty$. In \cite{bbl}, it was shown that for the Dirichlet form corresponding to the Cauchy process, the square function $G$ may be too large for certain $f \in \LpE$ with $1 < p < 2$; see \cite[Example~2]{bbl}.
As we will see later in Subsection~\ref{sub:H_func}, the square function $H$ is affected by the same problem; see Example~\ref{cex:Cauchy2} below. This is the reason for working with the operator $\Gampri$ instead of $\Gamma$.

However, the authors of \cite{bbl} demonstrated that $\normLp{G} \asymp \normLp{f}$ for $2 < p < \infty$ in the context of symmetric L{\'e}vy processes in the Euclidean space. These results were later extended to the non-symmetric case by Ba{\~n}uelos and Kim in \cite{bk}.

For the broader class of processes considered here, the lower bound for $\normLp{G}$ remains valid when $2 < p < \infty$. As shown in the next subsection, the Hardy\nobreakdash--Stein identity leads to the following result.
\begin{prop}
    Impose Assumption~\ref{ass:SS}.
    Let $2<p<\infty$ and $f\in\LpE$. Then,
    \begin{align*}
        \normLp{f} \lesssim \normLp{G} .
    \end{align*}
\end{prop}
The above results follow directly from $\widetilde{G} \le \sqrt{2} G$ and the estimate for $\widetilde{G}$, which is provided in Theorem~\ref{thm:Gprim_less_2inf} below.

\subsection{Square function $\widetilde{G}$}
\label{sub:Gprim_func}

Let us now consider the following square function:
\begin{align*}
    \widetilde{G}(x) & := 
    \left(
        \int\limits_0^\inft \Gampri[P_tf](x) \dts 
    \right)^{1/2}
    .
\end{align*}

The square function $\widetilde{G}$ overcomes the unboundedness issue for $1<p<2$ raised in Example~2 of \cite{bbl}. By employing the Hardy\nobreakdash--Stein identity, we deduce the following two estimates.

\begin{thm}
\label{thm:Gprim_less_12}
    Let $1<p\le2$ and $f\in\LpE$. Then,
    \begin{align*}
        \normLp{\widetilde{G}} \lesssim \normLp{f}.
    \end{align*}
\end{thm}
\begin{thm}
\label{thm:Gprim_less_2inf}
    Impose Assumption~\ref{ass:SS}.
    Let $2\le p<\infty$ and $f\in\LpE$. Then,
    \begin{align*}
        \normLp{f} \lesssim \normLp{\widetilde{G}} .
    \end{align*}
\end{thm}

In the next two proofs, we repeat mainly the argument used in \cite[Proofs of Lemma~4.3 and Lemma~4.5]{bbl}.

\begin{proof}[Proof of Theorem~\ref{thm:Gprim_less_12}.]
    Denote $f^*(x):=\sup_{t\ge0}|P_tf(x)|$. Since $p\le2$,
    \begin{align*}
         (\widetilde{G}(x))^2
         & =
         \int\limits_0^\inft \Gampri[P_tf](x) \dt
         \\
         & =
         (f^*(x))^{2 - p} \int\limits_0^\inft \Gampri[P_tf](x) |f^*(x)|^{p - 2} \dts
         \\
         & \le
         (f^*(x))^{2 - p} \int\limits_0^\inft \Gampri[P_tf](x) |P_tf(x)|^{p - 2} \dts .
    \end{align*}
    By Hölder's inequality, inequality \eqref{neq:HS_12_bezSS}, and Stein's maximal theorem~\eqref{neq:Stein},
    \begin{align*}
         \intE (\widetilde{G}(x))^p\,m(\dx)
         & \le
         \intE (f^*(x))^{(2 - p) p / 2} \left(\int\limits_0^\inft \Gampri[P_tf](x) |P_tf(x)|^{p - 2} \dts\right)^{p / 2}\,m(\dx)
         \\
         & \le
         \left(\intE (f^*(x))^p\,m(\dx)\right)^{1 - p / 2} \left(\int\limits_0^\inft  \intE \Gampri[P_tf](x) |P_tf(x)|^{p - 2}\,m(\dx)\dt \right)^{p / 2} \\
         & \lesssim
         \left(\intE (f^*(x))^p\,m(\dx)\right)^{1 - p / 2} \left(\intE |f(x)|^p\,m(\dx) \right)^{p / 2} \\
         & \asymp
         \intE |f(x)|^p\,m(\dx) .
    \end{align*}
    Therefore,
    \begin{align*}
     \intE (\widetilde{G}(x))^{p}\,m(\dx) & \lesssim \intE |f(x)|^p\,m(\dx) .
    \end{align*}
\end{proof}

\begin{proof}[Proof of Theorem~\ref{thm:Gprim_less_12}.]
    Denote $f^*(x):=\sup_{t\ge0}|P_tf(x)|$.
    Due to the strong stability of the semigroup, estimate~\eqref{sim:HS_2pinf} is valid.
    By this, Hölder's inequality, and Stein's maximal theorem~\eqref{neq:Stein}, we obtain
    \begin{align*}
     \intE |f(x)|^p\,m(\dx) & \asymp \int\limits_0^\inft  \intE \Gampri[P_tf](x) |P_tf(x)|^{p - 2}\,m(\dx)\dt \\
     & \le \int\limits_0^\inft \intE \Gampri[P_tf](x) |f^*(x)|^{p - 2}\,m(\dx)\dt  \\
     & = \intE (\widetilde{G}(x))^2 |f^*(x)|^{p - 2}\,m(\dx) \\
     & \le \left(\intE (\widetilde{G}(x))^{p}\,m(\dx)\right)^{2/p} \left(\intE |f^*(x)|^p\,m(\dx)\right)^{1 - 2/p} \\
     & \asymp \left(\intE (\widetilde{G}(x))^{p}\,m(\dx)\right)^{2/p} \left(\intE |f(x)|^p\,m(\dx)\right)^{1 - 2/p} .
    \end{align*}
    This proves the desired inequality:
    \begin{align*}
     \intE |f(x)|^p\,m(\dx) & \lesssim \intE (\widetilde{G}(x))^{p}\,m(\dx) .
    \end{align*}
\end{proof}

When the Dirichlet form corresponds to the pure\nobreakdash-jump L\'evy process, then the square function $\widetilde{G}$ possesses the both-side $L^p$-estimates for the entire range of $1<p<\infty$, as shown in Section~4 of \cite{bbl}.
However, this does not hold in general, as demonstrated in Example~\ref{cex:Brown} below.
Before we move on to this example, we would like to make some comments on the conclusions derived from it.
\begin{rem}
\label{rem:GlessH}
    The result below shows that the statement of Theorem~3.4 in \cite{lw} is not valid. The problem lies in the proof of Lemma~3.3, where the inequality $G(x)\le \sqrt{2}H(x)$ is stated.
    It was incorrectly assumed that $P_t[P_tf(k(\cdot,y))](x)$ is equal to $P_{2t}f(k(x,y))$ for some function $k$; see assumption~(A4) therein.
    This statement may be true, in the case of $E=\Rn$, $k(x+y)=x+y$, and translationally invariant semigroup $\Pt$, as shown in the proof of Lemma~4.2 in \cite{bbl}.
    Nevertheless, the following calculations confirm that this error is irreparable in general.
\end{rem}

\begin{ex}[Brownian motion on interval with removed segment]
\label{cex:Brown}
    The present example was proposed by Mateusz Kwaśnicki. The goal is to demonstrate that, for $2<p<\infty$, there is generally no uniform constant $C_p$ satisfying
    $\normLp{\widetilde{G}} \le C_p \normLp{f}$ for all $f\in\LpE$. Consequently, the same conclusion applies to the function $G$.

    \emph{Idea.}
        We construct a discrete approximation of the reflected Brownian motion on the interval $[0, \tfrac{\pi}{4}] \cup [\tfrac{3\pi}{4}, \pi]$ with additional jumps between $\tfrac{\pi}{4}$ and $\tfrac{3\pi}{4}$. The intensity of the jumps are chosen so that the function $\cos x$ becomes an eigenfunction of the generator.

        In other words, this process is an approximation of the reflected Brownian motion in $[0, \pi]$, but with segments of paths lying on the part $(\tfrac{\pi}{4}, \tfrac{3\pi}{4})$ deleted.

        By the approximation, we mean the following process:
        a Markov chain with continuous time
        which can be viewed as a reflected symmetric nearest-neighbor random walk $\Xt$ on $\{1, 2, \ldots, 4 n\}$, with parts of paths corresponding to $\{(t,X_t) : X_t \in \{n + 1, n + 2, \ldots, 3 n\}\}$ removed.

        To work with the standard notation for the generator matrix of a Markov chain, we reindex the state space:
        we shift the states $3 n + 1, 3 n + 2, \ldots, 4 n$ to the left and call them $n + 1, n + 2, \ldots, 2 n$.
        Therefore, $E = \{1, 2, \ldots, 2n\}$. The space $E$ is equipped with the counting measure $m$.

    \emph{Generator.}
        The generator of the process $\Xt$ is expressed as the following $2n \times 2n$ matrix $A=[a_{i,j}]_{i,j=1}^{2n}$:
        
        \begin{align*}
        A = \left(
        \setlength{\arraycolsep}{2pt}
        \begin{array}{*{14}{c}}
            -1 & 1 & 0 & 0 & \cdots & 0 & 0 & 0 & 0 & \cdots & 0 & 0 & 0 & 0 \\
            1 & -2 & 1 & 0 & \cdots & 0 & 0 & 0 & 0 & \cdots & 0 & 0 & 0 & 0 \\
            0 & 1 & -2 & 1 & \cdots & 0 & 0 & 0 & 0 & \cdots & 0 & 0 & 0 & 0 \\
            0 & 0 & 1 & -2 & \cdots & 0 & 0 & 0 & 0 & \cdots & 0 & 0 & 0 & 0 \\
            \vdots & \vdots & \vdots & \vdots & \ddots & \vdots & \vdots & \vdots & \vdots & \ddots & \vdots & \vdots & \vdots & \vdots \\
            0 & 0 & 0 & 0 & \cdots & -2 & 1 & 0 & 0 & \cdots & 0 & 0 & 0 & 0 \\
            0 & 0 & 0 & 0 & \cdots & 1 & \alpha^*_n & \alpha_n & 0 & \cdots & 0 & 0 & 0 & 0 \\
            0 & 0 & 0 & 0 & \cdots & 0 & \alpha_n & \alpha^*_n & 1 & \cdots & 0 & 0 & 0 & 0 \\
            0 & 0 & 0 & 0 & \cdots & 0 & 0 & 1 & -2 & \cdots & 0 & 0 & 0 & 0 \\
            \vdots & \vdots & \vdots & \vdots & \ddots & \vdots & \vdots & \vdots & \vdots & \ddots & \vdots & \vdots & \vdots & \vdots \\
            0 & 0 & 0 & 0 & \cdots & 0 & 0 & 0 & 0 & \cdots & -2 & 1 & 0 & 0 \\
            0 & 0 & 0 & 0 & \cdots & 0 & 0 & 0 & 0 & \cdots & 1 & -2 & 1 & 0 \\
            0 & 0 & 0 & 0 & \cdots & 0 & 0 & 0 & 0 & \cdots & 0 & 1 & -2 & 1 \\
            0 & 0 & 0 & 0 & \cdots & 0 & 0 & 0 & 0 & \cdots & 0 & 0 & 1 & -1
        \end{array}\right), 
        \end{align*}
        where
        \begin{align*}
            \alpha_n & := \frac{1}{1 + \cot \tfrac{\pi}{8 n}}
        \end{align*}
        and
        $\alpha^*_n:=-1-\alpha_n$.
        In other words, the jumping kernel is given by
        $J(x,\dy) = J(x,y) m(\dy)$ and
        $J(i,j) := a_{i,j}$ for $i\neq j$, where
        \begin{align*}
            a_{i,j}
            =
            \begin{cases}
                1 & \text{when } |i-j|=1 \text{ and } \{i,j\}\neq\{n,n+1\}, \\
                \alpha_n & \text{when } \{i,j\}=\{n,n+1\}, \\
                0 & \text{when } |i-j|>1 .
            \end{cases}
        \end{align*}
        The jumping density $J(x,y)$ is illustrated in Figure \ref{fig:jumps} below.

    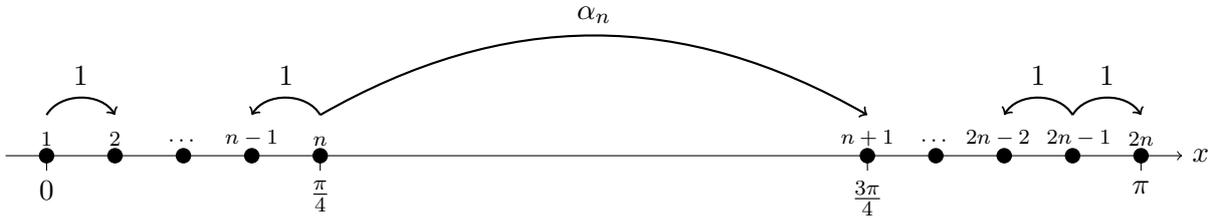
\begin{figure}[h]
    \centering
        \begin{tikzpicture}[scale=1.8]
        \small
            \draw[->] (-0.3,0) -- (8.3,0) node[right] {$x$};
        
            \draw (0,0) -- (0,-0.11) node[below] {$0$};
            \draw (2,0) -- (2,-0.11) node[below] {$\frac{\pi}{4}$};
            \draw (6,0) -- (6,-0.11) node[below] {$\frac{3\pi}{4}$};
            \draw (8,0) -- (8,-0.11) node[below] {$\pi$};
        
            \filldraw (0,0) circle (1.5pt) node[above] {\tiny $1$};
            \filldraw (0.5,0) circle (1.5pt) node[above] {\tiny $2$};
            \filldraw (1,0) circle (1.5pt) node[above] {\tiny $\cdots$};
            \filldraw (1.5,0) circle (1.5pt) node[above] {\tiny $n-1$};
            \filldraw (2,0) circle (1.5pt) node[above] {\tiny $n$};
        
            \filldraw (6,0) circle (1.5pt) node[above] {\tiny $n+1$};
            \filldraw (6.5,0) circle (1.5pt) node[above] {\tiny $\cdots$};
            \filldraw (7,0) circle (1.5pt) node[above] {\tiny $2n-2$\phantom{!!}};
            \filldraw (7.5,0) circle (1.5pt) node[above] {\tiny \phantom{!!}$2n-1$};
            \filldraw (8,0) circle (1.5pt) node[above] {\tiny $2n$};
        
            \draw[->, thick] (0,0.3) to[out=60,in=120] (0.5,0.3);
            \node[above] at (0.25,0.45) {$1$};
            
            \draw[->, thick] (2,0.3) to[out=30,in=150] (6,0.3);
            \node[above] at (4,0.9) {$\alpha_n$};

            \draw[->, thick] (2,0.3) to[out=120,in=60] (1.5,0.3);
            \node[above] at (1.75,0.45) {$1$};

            \draw[->, thick] (7.5,0.3) to[out=120,in=60] (7,0.3);
            \node[above] at (7.25,0.45) {$1$};

            \draw[->, thick] (7.5,0.3) to[out=60,in=120] (8,0.3);
            \node[above] at (7.75,0.45) {$1$};
            \normalsize
        \end{tikzpicture}
    \caption{The intensity of jumps of the process $\Xt$.}
    \label{fig:jumps}
    \end{figure}

    \emph{Eigenfunction.}
    We consider the following function on $E$:
    \begin{align*}
        f(k) & := \begin{cases}
          \cos \tfrac{(2 k - 1) \pi}{8 n} & \text{if } k \le n , \\
          \cos \tfrac{(2 k - 1 + 4 n) \pi}{8 n} & \text{if } k > n .
         \end{cases}
    \end{align*}

    Now, we will show that $f$ is an eigenfunction of $A$, namely
    \begin{align*}
        A f(k) = -\lambda_n f(k),
    \end{align*}
    where
    \begin{align*}
        \lambda_n := 2 (1 - \cos \tfrac{\pi}{4 n}) = 4 \sin^2 \tfrac{\pi}{8 n} .
    \end{align*}
    Observe that $f(2 n + 1 - k) = -f(k)$ and $A f(2 n + 1 - k) = -A f(k)$, hence it is enough to consider $k \le n$.
    By the sum-to-product formula, for $k = 2, 3, \ldots, n - 1$, we get
    \begin{align*}
        A f(k) & = f(k + 1) + f(k - 1) - 2 f(k) \\
        & = \cos \tfrac{(2 k + 1) \pi}{8 n} + \cos \tfrac{(2 k - 3) \pi}{8 n} - 2 \cos \tfrac{(2 k - 1) \pi}{8 n} \\
        & = 2 \cos \tfrac{(2 k - 1) \pi}{8 n} \cos \tfrac{\pi}{4 n} - 2 \cos \tfrac{(2 k - 1) \pi}{8 n} = -\lambda_n f(k) .
    \end{align*}
    Similarly, for $k = 1$
    \begin{align*}
        A f(1) & = f(2) - f(1) \\
        & = \cos \tfrac{3 \pi}{8 n} - \cos \tfrac{\pi}{8 n} \\
        & = \cos \tfrac{3 \pi}{8 n} + \cos \tfrac{-\pi}{8 n} - 2 \cos \tfrac{\pi}{8 n}
        \\
        &=
        2\cos\tfrac{\pi}{8n}\cos\tfrac{\pi}{4n}
        -
        2 \cos \tfrac{\pi}{8 n}.
    \end{align*}
    Thus, again $A f(1) = -\lambda_n f(1)$. Finally, when $k = n$, we have
    \begin{align*}
        A f(n) & = \alpha_n (f(n + 1) - f(n)) + (f(n - 1) - f(n)) \\
        & = \alpha_n (\cos \tfrac{(6 n + 1) \pi}{8 n} - \cos \tfrac{(2 n - 1) \pi}{8 n}) + \cos \tfrac{(2 n - 3) \pi}{8 n} - \cos \tfrac{(2 n - 1) \pi}{8 n} \\
        & = -2 \alpha_n \cos \tfrac{(2 n - 1) \pi}{8 n} + 2 \sin \tfrac{(n - 1) \pi}{4 n} \sin \tfrac{\pi}{8 n} \\ 
        & = 2 \left(-\alpha_n + \frac{\sin \tfrac{(n - 1) \pi}{4 n} \sin \tfrac{\pi}{8 n}}{\cos \tfrac{(2 n - 1) \pi}{8 n}}\right) \cos \tfrac{(2 n - 1) \pi}{8 n} .
    \end{align*}
    Note that the parameter $\alpha_n$ is chosen so that the right-hand side is equal to $-\lambda_n f(n)$. Indeed,
    \begin{align*}
        & \left(-\alpha_n + \frac{\sin \tfrac{(n - 1) \pi}{4 n} \sin \tfrac{\pi}{8 n}}{\cos \tfrac{(2 n - 1) \pi}{8 n}} \right) \\
        & \qquad = \frac{-\sin \tfrac{\pi}{8 n}}{\sin \tfrac{\pi}{8 n} + \cos \tfrac{\pi}{8 n}} + \frac{\tfrac{\sqrt{2}}{2} (\cos \tfrac{\pi}{4 n} - \sin \tfrac{\pi}{4 n}) \sin \tfrac{\pi}{8 n}}{\tfrac{\sqrt{2}}{2} (\cos \tfrac{\pi}{8 n} + \sin \tfrac{\pi}{8 n})} \\
        & \qquad = -\frac{(1 - \cos \tfrac{\pi}{4 n} + \sin \tfrac{\pi}{4 n}) \sin \tfrac{\pi}{8 n}}{\cos \tfrac{\pi}{8 n} + \sin \tfrac{\pi}{8 n}} \\
        & \qquad = -\frac{(2 \sin^2 \tfrac{\pi}{8 n} + 2 \sin \tfrac{\pi}{8 n} \cos \tfrac{\pi}{8 n}) \sin \tfrac{\pi}{8 n}}{\cos \tfrac{\pi}{8 n} + \sin \tfrac{\pi}{8 n}} \\
        & \qquad = -2 \sin^2 \tfrac{\pi}{8 n} = -\frac{\lambda_n}{2} \, .
    \end{align*}
        
    \emph{Square functions.}
    According to the earlier part, since $f$ is an eigenfunction of the generator $A$,
    the semigroup operator $P_t$ acting on $f$ is given by
    $P_t f(k) = e^{-\lambda_n t} f(k)$.
    Recall that the square function $\widetilde{G}$ for the function $f$ is given by
    \begin{align*}
        (\widetilde{G}(x))^2 = \int\limits_0^\infty \intEx (P_t f(y) - P_t f(x))^2 \chi(P_tf(x),P_tf(y))
        \,J(x, y) m(\dy) \dt.
    \end{align*}
    Since $\int_0^\inft e^{-\lambda_n t} \dts= 1/\lambda_n$ and $m$ is the counting measure, we may write
    \begin{align*}
        (\widetilde{G}(k))^2 & = \frac{1}{\lambda_n} \sum_{l \ne k} (f(l) - f(k))^2 \chi(f(k),f(l)) J(k, l) .
    \end{align*}
    Moreover, $\widetilde{G}(2 n + 1 - k) = \widetilde{G}(k)$.
    When $k = 1, 2, \ldots, n - 1$, we have $0<f(k+1)<f(k)$ and
    \begin{align*}
        (\widetilde{G}(k))^2 & = \frac{1}{\lambda_n} (f(k + 1) - f(k))^2 = \frac{4 (\sin \tfrac{k \pi}{4 n} \sin \tfrac{\pi}{8 n})^2}{4 \sin^2 \tfrac{\pi}{8 n}} = \sin^2 \tfrac{k \pi}{4 n} .
    \end{align*}
    Also, if $k = n$, then $|f(n)|=|f(n+1)|=-f(n)$, hence
    \begin{align*}
        (\widetilde{G}(n))^2 & = \frac{\alpha_n}{2 \lambda_n}  (f(n + 1) - f(n))^2 \\
        & =
        \frac{(2 \cos \tfrac{(2 n - 1) \pi}{8 n})^2}{8 \sin^2 \tfrac{\pi}{8 n} (1 + \cot \tfrac{\pi}{8 n})} \\
        & = \frac{(\sqrt{2} (\cos \tfrac{\pi}{8 n} + \sin \tfrac{\pi}{8 n}))^2}{8 \sin \tfrac{\pi}{8 n} (\sin \tfrac{\pi}{8 n} + \cos \tfrac{\pi}{8 n})} \\
        & = \frac{\sin \tfrac{\pi}{8 n} + \cos \tfrac{\pi}{8 n}}{4 \sin \tfrac{\pi}{8 n}} = \frac{1 + \cot \tfrac{\pi}{8 n}}{4} \, .
    \end{align*}
    Now, the $p$\nobreakdash-norm of $\widetilde{G}$ is given by
    \begin{align}
    \label{eq:cex_mateusz}
        \normLp{\widetilde{G}}^p & = 2 \sum_{k = 1}^{n - 1} \sin^p \tfrac{k \pi}{4 n} + 2 (\widetilde{G}(n))^p .
    \end{align}
    The first term may be written as
    \begin{align*}
        2 \sum_{k = 1}^{n - 1} \sin^p \tfrac{k \pi}{4 n}
        =
        2 n \sum_{k = 1}^{n} \frac{1}{n}\sin^p \tfrac{k \pi}{4 n} - 2\left(\frac{\sqrt{2}}{2}\right)^p.
    \end{align*}
    Note, that the sum on the right\nobreakdash-hand side is the Riemann sum of the function $2\sin^p \tfrac{\pi x}{4}$ over $[0,1]$.
    Therefore, the first term on the right\nobreakdash-hand side of \eqref{eq:cex_mateusz} is asymptotically equal to $c_1 n$, where $c_1 := 2 \int_0^1 \sin^p \tfrac{\pi x}{4} \dxs > 0$.

    The second term in \eqref{eq:cex_mateusz} is asymptotically equal to $2 (\tfrac{1}{4} \cdot \tfrac{8 n}{\pi})^{p/2}=2(\tfrac{2n}{\pi})^{p/2}$.
    Hence, for $p > 2$,
    \begin{align*}
         \normLp{\widetilde{G}}^p & \sim \frac{2^{1 + p/2}}{\pi^{p/2}} \, n^{p/2}.
    \end{align*}
    Particularly, we see that $\widetilde{G}(n)$ is the dominating term in the $p$\nobreakdash-norm of $\widetilde{G}$.
    On the other hand,
    \begin{align*}
         \normLp{f}^p & = 2 \sum_{k = 1}^n \cos^p \tfrac{k \pi}{4 n} \sim c_2 n ,
    \end{align*}
    where $c_2 := 2 \int_0^1 (\cos \tfrac{\pi x}{4})^p \dxs > 0$. Finally,
    \begin{align*}
        \frac{\normLp{\widetilde{G}}}{\normLp{f}} & \sim \frac{2^{1/p + 1/2}}{\pi^{1/2} c_2^{1/p}} \, n^{1/2 - 1/p} .
    \end{align*}
    This proves that for $2<p<\infty$ there is no universal constant $C_p>0$ such that
    \begin{align*}
        \normLp{\widetilde{G}} & \le C_p \normLp{f} .
    \end{align*}
    What is more, since $\widetilde{G} \le \sqrt{2} G$,  the above statement holds also for the square function $G$.
\end{ex}

\subsection{Square function $H$}
\label{sub:H_func}

The next square function we will analyze is the following:
\begin{align*}
    H(x) & :=
    \left(
        \int\limits_0^\inft P_t \Gamma [P_tf](x) \dts 
    \right)^{1/2}
    .
\end{align*}

The function $H$ overcomes the problems with $\widetilde{G}$ demonstrated for $2<p<\inft$ in Example~\ref{cex:Brown}.
Indeed, the following $L^p$-estimates hold.
\begin{thm}
\label{thm:H_less_2inf}
    Impose Assumption~\ref{ass:coserv}. Let $2\le p<\infty$ and $f\in\LpE$. Then,
    \begin{align*}
        \normLp{H} \lesssim \normLp{f}.
    \end{align*}
\end{thm}
\begin{thm}
\label{thm:H_great_12}
    Impose Assumptions~\ref{ass:coserv} and \ref{ass:SS}.
    Let $1<p\le 2$ and $f\in\LpE\cap\LE{2}$. Then,
    \begin{align*}
        \normLp{f} \lesssim \normLp{H}.
    \end{align*}
\end{thm}
\begin{prop}
\label{thm:H_great_3inf}
    Impose Assumption~\ref{ass:SS}.
    Let $3\le p<\infty$ and $f\in\LpE$. Then,
    \begin{align*}
        \normLp{f} \lesssim \normLp{H}.
    \end{align*}
\end{prop}

The last statement is a straightforward consequence of Proposition~\ref{prop:Hprim_great_3inf} below, due to $\widetilde{H} \le \sqrt{2} H$.
The rest of the estimates can be shown employing the martingale $\Mt$ introduced in Subsection~\ref{sub:martingale} and the Burkholder--Davies--Gundy inequality; see the proofs at the end of the current subsection.
However, the square function $H$ shares the same issue as $G$, namely, since it is defined by the operator $\Gamma$ instead of $\Gampri$, it is too large when $1<p<2$. To demonstrate this, we adopt Example~2 of \cite{bbl}, see Example~\ref{cex:Cauchy2} below.
Moreover, the lower bound of $\normLp{H}$ is still an open question for the range $2<p<3$.

\begin{ex}
\label{cex:Cauchy2}
    Consider the Dirichlet form on the Euclidean space $E = \Rn$ ($n\ge2$) associated with the Cauchy process ($\alpha$-stable process with $\alpha=1$), that is, $J(x, \dy) = J(x,y)\dy$, where
    \begin{align*}
        J(x, y) = \frac{\mathcal{A}_{n,1}}{|y - x|^{n + 1}}.
    \end{align*}
    Here
    $\mathcal{A}_{n,\alpha} := 2^\alpha \Gamma((n+\alpha)/2) \pi^{-n/2}/\abs{\Gamma(-\alpha/2)}$,
    where $\alpha\in(0,2)$;
    see \cite[(2.11)]{bbl}.
    In this context, $\Gamma(\cdot)$ stands for the Gamma function.
    The semigroup is given by the following transition density:
    \begin{align*}
        p_t(x,y) = \frac{c_n t}{(t^2+|y-x|^2)^{(n+1)/2}}
    \end{align*}
    and $c_n = \Gamma((n+1)/2) \pi^{-(n+1)/2}$.

    Denote the ball of radius $r$ centered at a point $x$ by $B(x,r)$.
    Let $h(x):=\abs{x}^{-(n+1)/2}$ and $f(x):=h(x)\ind_{B(0,1)}(x)$.
    Note that $f$ belongs to $L^p(\dx)$ for $1<p<2n/(n+1)<2$.
    We claim that the square function $H$ of the function $f$ is identically equal to infinity.
    The problem fundamentally arises from the rapid growth of $\Gamma [P_tf]$ as $t$ tends to zero.

    To illustrate this, we introduce the following function:
    \begin{align*}
        v(x,t) := \begin{cases}
		P_th(x), & \text{for } t>0,\\
		h(x), & \text{for } t=0.
	\end{cases}
    \end{align*}
    Note that $h$ is locally integrable on $\Rn$ and vanishes at infinity, hence the function $v$ is well\nobreakdash-defined and continuous except at the point $(x,t)=(0,0)$.

    We first calculate $P_tf(x)$. Observe that $p_t(x,y)$ satisfies the following scaling property
    \begin{align*}
        p_t(x,y) = \frac{1}{t^n} p_1\!\left( \tfrac{x}{t}, \tfrac{y}{t} \right).
    \end{align*}
    Obviously $h(x) = t^{-(n+1)/2}h(x/t)$.
    Therefore,
    \begin{align*}
        P_tf(x) &=
        c_n \int\limits_{\mathclap{B(0,1)}} \frac{1}{t^{(n+1)/2}}\, h\!\left(\tfrac{y}{t}\right)\cdot \frac{1}{t^n}\,p_1\!\left( \tfrac{x}{t}, \tfrac{y}{t} \right)\dys
        \\
        &=
        c_n \int\limits_{\mathclap{B(0,1/t)}} \frac{1}{t^{(n+1)/2}}\, h\!\left(y\right)\cdot \frac{1}{t^n}\,p_1\!\left( \tfrac{x}{t}, y \right) t^n\dys
        \\
        &=
        t^{-(n+1)/2} v_t\!\left( \tfrac{x}{t}, 1  \right),
    \end{align*}
    where
    \begin{align*}
        v_s(x,t) :=
        \int\limits_{\mathclap{B(0,1/s)}}
        h(y)p_t(x,y) \dys.
    \end{align*}
    Analogously,
    \begin{align*}
        v(x,t) = t^{-(n+1)/2} v\!\left( \tfrac{x}{t}, 1  \right),
        \quad
        t>0,\,x\in\Rn.
    \end{align*}

    Now we can rewrite the square function of $f$ as follows:
    \begin{align}
    \label{eq:cex:GCauchy2}
        (H(x))^2
        &=
        \mathcal{A}_{n,1}
        \int\limits_0^{\infty}\int\limits_{\Rn}\int\limits_{\Rn}
        \frac{(P_tf(y)-P_tf(z))^2}{|z-y|^{n+1}}
        \,
        p_t(x,z)
        \dys\dz\dt
        \notag\\	
        &=
        \mathcal{A}_{n,1}
        \int\limits_0^{\infty}\int\limits_{\Rn}\int\limits_{\Rn}
        \frac{1}{t^{n+1}}
        \cdot
        \frac{(v_t(\tfrac{y}{t},1)-v_t(\tfrac{z}{t},1))^2}{|z-y|^{n+1}}
        \,
        p_t(x,z)
        \dys\dz\dt
        \notag\\	
        &=
        \mathcal{A}_{n,1}
        \int\limits_0^{\infty} \int\limits_{\Rn} \int\limits_{\Rn}
        \frac{1}{t}
        \cdot
        \frac{(v_t(y,1)-v_t(\tfrac{z}{t},1))^2}{|z-ty|^{n+1}}
        \,
        p_t(x,z)
        \dys\dz\dt
        \notag\\	
        &=
        \mathcal{A}_{n,1}
        \int\limits_0^{\infty} \int\limits_{\Rn} \int\limits_{\Rn}
        t^{n-1}
        \cdot
        \frac{(v_t(y,1)-v_t(z,1))^2}{|tz-ty|^{n+1}}
        \,
        p_t(x,tz)
        \dys\dz\dt
        \notag\\	
        &=
        \mathcal{A}_{n,1} c_n
        \int\limits_{\Rn} \int\limits_{\Rn}
        \int\limits_0^{\infty}
        \frac{1}{t}
        \cdot
        \frac{(v_t(y,1)-v_t(z,1))^2}{|z-y|^{n+1} (t^2+\abs{tz-x}^2)^{(n+1)/2}}
        \dts\dy\dz.
    \end{align}
    In the last line, we used Tonelli's theorem.
    Observe that $v_t(y,1)\to v(y,1)>0$ as $t\to0^+$. Therefore, we have the following asymptotic equivalence of the integrand. For $x\neq0$,
    \begin{align*}
        \frac{1}{t}
        \cdot
        \frac{(v_t(y,1)-v_t(z,1))^2}{|z-y|^{n+1} (t^2+\abs{tz-x}^2)^{(n+1)/2}}
        \sim
        \frac{(v(y,1)-v(z,1))^2}{|z-y|^{n+1}\abs{x}^{n+1}}
        \cdot
        \frac{1}{t},
        \quad
        \text{as }
        t\to0^+,
    \end{align*}
    and when $x=0$,
    \begin{align*}
        \frac{1}{t}
        \cdot
        \frac{(v_t(y,1)-v_t(z,1))^2}{|z-y|^{n+1} (t^2+\abs{tz-x}^2)^{(n+1)/2}}
        \sim
        \frac{(v(y,1)-v(z,1))^2}{|z-y|^{n+1}(1+\abs{z}^2)^{(n+1)/2}}
        \cdot
        \frac{1}{t^{n+2}},
        \quad
        \text{as }
        t\to0^+.
    \end{align*}
    In both cases, the inner integral of \eqref{eq:cex:GCauchy2} diverges to infinity. Therefore, $H\equiv\inft$.
\end{ex}

\begin{proof}[Proof of Theorem~\ref{thm:H_less_2inf}]
    Initially, we assume that $f\in\DE$. For $p\ge2$, Jensen's inequality yields
    \begin{align*}
     \intE (H(x))^{p}\,m(\dx)
     & =
     \intE \left(\int\limits_0^\inft P_t\Gamma[P_tf](x)  \dts\right)^{p/2}\,m(\dx)
     \\
     & =
     \intE \left(\Ev_x \int\limits_0^\inft \Gamma[P_tf](X_t) \dts\right)^{p/2}\,m(\dx)
     \\
     & \le
     \intE \Ev_x \left(\int\limits_0^\inft \Gamma[P_tf](X_t) \dts\right)^{p/2}\,m(\dx)
     \\
     & =
     \lim_{T \to \inft} \intE \Ev_x \left(\int\limits_0^T \Gamma[P_tf](X_t) \dts\right)^{p/2}\,m(\dx) .
    \end{align*}
    In the last line, we used the monotone convergence theorem.
    Since the semigroup is conservative, we can utilize self-duality of the process $\Xt$ (see~\eqref{eq:duality}) to write
    \begin{align*}
     \intE (H(x))^{p}\,m(\dx) & \le \lim_{T \to \inft} \intE \Ev_x \left(\int\limits_0^T \Gamma[P_{T-t}f](X_t) \dt\right)^{p/2}\,m(\dx) .
    \end{align*}
    Let $\Mt$ be the martingale given by \eqref{eq:Mt}.
    By the assumption $f\in\DE$ and Corollary~\ref{cor:sharpMt},
    \begin{align*}
        \intE \Ev_x \left(\int\limits_0^T \Gamma[P_{T-t}f](X_t) \dt\right)^{p/2}\,m(\dx) =
        2^{-p/2}
        \intE \Ev_x (\langle M \rangle_T)^{p/2}\,m(\dx).
    \end{align*}
    Employing the Burkholder--Davies--Gundy inequality~\eqref{sim:bdg}, together with inequality \eqref{eq:brackets-}, we obtain
    \begin{align*}
         \intE (H(x))^{p}\,m(\dx)
         & \lesssim
         \lim_{T \to \inft} \intE \Ev_x (\langle M \rangle_T)^{p/2}\,m(\dx) \\
         & \lesssim
         \lim_{T \to \inft} \intE \Ev_x ([M]_T)^{p/2}\,m(\dx) \\
         & \lesssim
         \lim_{T \to \inft} \intE \Ev_x |M_T|^p\,m(\dx) \\
         & = \lim_{T \to \inft} \intE \Ev_x |f(X_T) - P_Tf(x)|^p\,m(\dx) .
    \end{align*}
    Since $|a + b|^p \le 2^{p - 1} (|a|^p + |b|^p)$, we get
    \begin{align*}
         \intE (H(x))^{p}\,m(\dx)
         & \lesssim
         \lim_{T \to \inft} \intE (\Ev_x |f(X_T)|^p + |P_Tf(x)|^p)\,m(\dx) \\
         & = \lim_{T \to \inft} \intE (P_T [|f|^p](x) + |P_Tf(x)|^p)\,m(\dx).
    \end{align*}
    Due to the contraction property of $\Pt$, it follows that $\normLp{P_Tf}^p\le \normLp{f}^p$.
    By the conservativeness of $\Pt$, we have
    $\int_E P_T [|f|^p](x)\,m(\dx)=\int_E |f(x)|^p\,m(\dx)$.
    Summarizing,
    \begin{align*}
        \intE (H(x))^{p}\,m(\dx)  & \lesssim \intE |f(x)|^p\,m(\dx).
    \end{align*}

    At this point, we relax the assumption $f\in\DE$.
    Let $f$ be an arbitrary function in $\LpE$.
    Since the Dirichlet form is regular, we may take the sequence $(f_n)$ of compactly supported functions from $\DE$ such that it converges almost everywhere to $f$.

    By $H[f_n]$ we denote the square function $H$ applied to $f_n$.
    Then, by Fatou's lemma,
    \begin{align*}
        \intE (H(x))^{p}\,m(\dx)
        & \leq
        \lim_{n\to\inft} \intE (H[f_n](x))^{p}\,m(\dx)
        \\
        & \lesssim
        \lim_{n\to\inft} \intE |f_n(x)|^p\,m(\dx)
        =
        \intE |f(x)|^p\,m(\dx).
    \end{align*}
\end{proof}

\begin{proof}[Proof of Theorem~\ref{thm:H_great_12}]
    As in the previous proof, we start with $f\in\DE$. We now consider the case $p<2$; hence by Jensen's inequality, we obtain
    \begin{align*}
         \intE (H(x))^{p}\,m(\dx) & = \intE \left(\int\limits_0^\inft P_t \Gamma[P_tf](x) \dts\right)^{p/2}\,m(\dx)
         \\
         & =
         \intE \left(\Ev_x \int\limits_0^\inft \Gamma[P_tf](X_t) \dts\right)^{p/2}\,m(\dx)
         \\
         & \ge
         \intE \Ev_x \left(\int\limits_0^\inft \Gamma[P_tf](X_t) \dts\right)^{p/2}\,m(\dx)
         \\
         & =
         \lim_{T \to \inft} \intE \Ev_x \left(\int\limits_0^T \Gamma[P_tf](X_t) \dts\right)^{p/2}\,m(\dx) .
    \end{align*}
    The last line follows from the monotone convergence theorem.
    By self-duality of the process $\Xt$ (formula~\eqref{eq:duality}),
    \begin{align*}
     \intE (H(x))^{p}\,m(\dx) & \ge \lim_{T \to \inft} \intE \Ev_x \left(\int\limits_0^T \Gamma[P_{T-t}f](X_t) \dts\right)^{p/2}\,m(\dx) .
    \end{align*}
    We again employ the martingale $\Mt$ given by \eqref{eq:Mt}.
    Given that $f\in\DE$, we deduce from Corollary~\ref{cor:sharpMt} that
    \begin{align*}
        \intE \Ev_x \left(\int\limits_0^T \Gamma[P_{T-t}f](X_t) \dt\right)^{p/2}\,m(\dx) =
        2^{-p/2}
        \intE \Ev_x (\langle M \rangle_T)^{p/2}\,m(\dx),
    \end{align*}
    By the Burkholder--Davies--Gundy inequality~\eqref{sim:bdg}, together with inequality \eqref{eq:brackets+}, we get
    \begin{align*}
         \intE (H(x))^{p}\,m(\dx)
         & \gtrsim
         \lim_{T \to \inft} \intE \Ev_x (\langle M \rangle_T)^{p/2}\,m(\dx)
         \\
         & \gtrsim
         \lim_{T \to \inft} \intE \Ev_x ([M]_T)^{p/2}\,m(\dx)
         \\
         & \gtrsim
         \lim_{T \to \inft} \intE \Ev_x |M_T|^p\,m(\dx)
         \\
         & =
         \lim_{T \to \inft} \intE \Ev_x |f(X_T) - P_Tf(x)|^p\,m(\dx) .
    \end{align*}
    By the elementary estimate $|a + b|^p \ge 2^{1 - p} |a|^p - |b|^p$, we get
    \begin{align*}
         \intE (H(x))^{p}\,m(\dx) 
         & \gtrsim
         \lim_{T \to \inft} \intE (\Ev_x |f(X_T)|^p - 2^{p - 1} |P_Tf(x)|^p)\,m(\dx)
         \\
         & =
         \lim_{T \to \inft} \intE (P_T [|f|^p](x) - 2^{p - 1}  |P_Tf(x)|^p)\,m(\dx).
    \end{align*}
    By the conservativeness of $\Pt$, we have
    $\int_E P_T [|f|^p](x)\,m(\dx)=\int_E |f(x)|^p\,m(\dx)$.
    Moreover, under Assumption~\ref{ass:SS},
    $\lim_{T \to \inft} \int_E |P_Tf(x)|^p\,m(\dx) = 0$. Hence, finally,
    \begin{align*}
     \intE (H(x))^{p}\,m(\dx)  & \gtrsim \intE |f(x)|^p\,m(\dx) .
    \end{align*}

    We relax the assumption $f\in\DE$, consider an arbitrary $f\in\LpE\cap\LE{2}$, and approximate $f$ by $P_sf$ for some $s>0$.
    Observe that since $f\in\LE{2}$, we have $P_sf\in\DE$.
    
    Denote by $H[P_sf]$ the square function $H$ acting on $P_sf$, namely, $H[P_sf]$ is given by
    \begin{align*}
        H[P_sf](x) = \left(
        \int\limits_0^\inft P_t \Gamma [P_tP_sf](x) \dts 
        \right)^{1/2}
        =
        \left(
        \int\limits_0^\inft P_t \Gamma [P_{s+t}f](x) \dts 
        \right)^{1/2}
        .
    \end{align*}
    By the monotone convergence theorem,
    \begin{align}
    \label{eq:PsH[Psf]}
        (H(x))^2
        &=
        \lim_{s\to0^+} \int\limits_s^\inft P_t \Gamma [P_tf](x) \dts
        =
        \lim_{s\to0^+} \int\limits_0^\inft P_sP_t \Gamma [P_tP_sf](x) \dts
        \notag \\
        &=
        \lim_{s\to0^+} P_s(H[P_sf]^2)(x).
    \end{align}
    
    Since the result holds for $P_sf$,
    \begin{align*}
        \intE (H[P_sf](x))^{p}\,m(\dx)  & \gtrsim \intE |P_sf(x)|^p\,m(\dx).
    \end{align*}
    Combining this with the monotone convergence theorem and Jensen's inequality, we obtain
    \begin{align*}
        \intE (H(x))^{p}\,m(\dx) 
        &=
        \lim_{s\to0^+}
        \intE \left(P_s(H[P_sf]^2)(x)\right)^{p/2}\,m(\dx)
        \ge
        \lim_{s\to0^+}
        \intE (H[P_sf](x))^{p}\,m(\dx)
        \\
        & \gtrsim
        \lim_{s\to0^+}
        \intE |P_sf(x)|^p\,m(\dx)
        =
        \intE |f(x)|^p\,m(\dx).
    \end{align*}
    The last line follows from
    the strong continuity of the semigroup $\Pt$.
\end{proof}

\subsection{Square function $\widetilde{H}$}
\label{sub:Hprim_func}

The last subsection is dedicated to the last square function:
\begin{align*}
    \widetilde{H}(x) & =
    \left(
        \int\limits_0^\inft P_t \Gampri[P_tf](x)\dts 
    \right)^{1/2}.
\end{align*}

We conjecture that the above function resolves the issue of $L^p$-unboundedness, which affects $\widetilde{G}$ for $p>2$ (see Example~\ref{cex:Brown}) and $H$ for $1<p<2$ (see Example~\ref{cex:Cauchy2}).
However, the method for establishing Littlewood--Paley estimates for this function remains unclear at this point.
Methods based on the Hardy--Stein identity seem to be well-suited for the functions $G$ and $\widetilde{G}$,
whereas martingale methods appear to be more applicable to the function $H$.

On the one hand, we are able to obtain the upper bound directly from Theorem~\ref{thm:H_less_2inf} and the fact that $\widetilde{H} \le \sqrt{2} H$.
\begin{prop}
\label{prop:Hprim_less_2inf}
    Impose Assumption~\ref{ass:coserv}. Let $2\le p<\infty$ and $f\in\LpE$. Then,
    \begin{align*}
        \normLp{\widetilde{H}} \lesssim \normLp{f}.
    \end{align*}
\end{prop}
On the other hand, when $p$ is sufficiently large for the estimate $|P_t f|^{p - 2} \le P_t(|f|^{p - 2})$ a.e. to hold, one may adapt the procedure from the proof of Theorem~\ref{thm:Gprim_less_12}, which is based on the Hardy--Stein identity, to obtain the lower bound of $\widetilde{H}$.
Nevertheless, it remains an open question whether this statement holds for $p\in(2,3)$.
\begin{prop}
\label{prop:Hprim_great_3inf}
    Impose Assumption~\ref{ass:SS}.
    Let $3\le p<\infty$ and $f\in\LpE$. Then,
    \begin{align*}
        \normLp{f} \lesssim \normLp{\widetilde{H}}.
    \end{align*}
\end{prop}
\begin{proof}
    Since $p \ge 3$, the function $a\mapsto |a|^{p-2}$ is convex. Thus, by Jensen's inequality,
    $|P_t f|^{p - 2} \le P_t(|f|^{p - 2})$ a.e.
    Therefore, by~\eqref{sim:HS_2pinf}, the symmetry of the operator $P_t$, and Hölder's inequality
    \begin{align*}
     \intE |f(x)|^p\,m(\dx)
     & \asymp
     \int\limits_0^\inft  \intE \Gampri[P_tf](x) |P_t f(x)|^{p - 2} \,m(\dx)\dt
     \\
     & \le
     \int\limits_0^\inft  \intE \Gampri[P_tf](x) P_t(|f|^{p - 2})(x)\,m(\dx)\dt
     \\
     & =
     \int\limits_0^\inft
     \intE P_t \Gampri[P_tf](x) |f(x)|^{p - 2}\,m(\dx)\dt
     \\
     & =
     \intE (\widetilde{H}(x))^2 |f(x)|^{p - 2}\,m(\dx)
     \\
     & \le
     \biggl(\intE (\widetilde{H}(x))^{p}\,m(\dx)\biggr)^{2/p} \biggl(\intE |f(x)|^p\,m(\dx)\biggr)^{1 - 2/p} .
    \end{align*}
    This implies
    \begin{align*}
     \intE |f(x)|^p\,m(\dx) & \lesssim \intE (\widetilde{H}(x))^{p}\,m(\dx) .
    \end{align*}
\end{proof}

\subsection*{Acknowledgements}

The author is grateful to Mateusz Kwaśnicki for a helpful discussion, numerous suggestions, and the construction of Example~\ref{cex:Brown}. He is especially thankful for Mateusz's generous support during the final stage of preparing this paper.
The author also wants to thank Błażej Wróbel for helpful comments on the constant $D_p$ in Stein's maximal theorem.

This research was funded in whole or in part by National Science Centre, Poland, 2023/49/B/ST1/04303. For the purpose of Open Access, the authors have applied a~CC\nobreakdash-BY public copyright licence to any Author Accepted Manuscript (AAM) version arising from this submission.

\bibliographystyle{plain}
\bibliography{biblio_d}

\end{document}